\documentclass[a4paper,10pt]{article}
\usepackage[T1]{fontenc}
\usepackage{amsmath, amssymb}
\usepackage[english]{babel}
\usepackage[utf8x]{inputenc} 
\usepackage{amsthm}
\usepackage[titletoc]{appendix}
\usepackage{amsmath}
\usepackage{amsfonts}
\usepackage{amssymb}
\DeclareMathOperator{\Ima}{im}
\usepackage{enumitem}
\usepackage{calligra}
\usepackage{calrsfs}
\usepackage[all]{xy}
{\theoremstyle {remark} }
{\theoremstyle {remark} }
{\theoremstyle {definition} }
{\theoremstyle {definition} \newtheorem {defa} {Definition} [section]}
{\theoremstyle {plain}  \newtheorem {thm} [defa] {Theorem}}
{\theoremstyle {plain}  \newtheorem {cor} [defa]{Corollary}}
{\theoremstyle {plain} \newtheorem {prop} [defa]{Proposition}}
{\theoremstyle {plain} \newtheorem {lem}[defa] {Lemma}}
{\theoremstyle {remark} \newtheorem {rem}[defa] {Remark}}
{\theoremstyle {remark} \newtheorem {ex}[defa] {Example}}
\usepackage[mathscr]{eucal}
\usepackage{tikz-cd}
\usetikzlibrary{cd}
\usepackage{tikz}
\usetikzlibrary{matrix}

\begin{document}

\title{On the prequantisation map for 2-plectic manifolds}
\author{Gabriel Sevestre\thanks{gabriel.sevestre@univ-lorraine.fr}~  
 and Tilmann Wurzbacher\thanks{tilmann.wurzbacher@univ-lorraine.fr}\\
\phantom{Higgs} \\
Institut \'Elie Cartan Lorraine \\
Universit\'{e} de Lorraine et C.N.R.S. \\
F-57000 Metz, France \\}

\date{\today}

\maketitle

\centerline{\it{In memoriam Kirill Mackenzie}}

\begin{abstract}
\noindent For a manifold $M$ with an integral closed 3-form $\omega$, we construct a $PU(H)$-bundle and a Lie groupoid over its total space,
together with a curving in the sense of gerbes. If the form is non-degenerate, we furthermore give a natural Lie 2-algebra quasi-isomorphism 
from the observables of $(M,\omega)$ to the weak symmetries of the above geometric structure, generalising the prequantisation
map of Kostant and Souriau.
\end{abstract}


integral three-forms, multiplicative vector fields

\tableofcontents

\section*{Introduction}
\addcontentsline{toc}{section}{Introduction}

Geometric quantisation of symplectic manifolds gives an excellent guideline and many rigorous results on how to go from 
classical to quantum mechanics. 
Its procedure for a manifold equipped with an integral symplectic form $\omega$ can be divised in four steps. First a geometric set-up is 
established from the data: a principal $S^1$-bundle $q: P\to M$ with connection 1-form $A$ whose curvature $F^A=dA$ equals 
$(-i2\pi) \omega$.
Then a Lie algebra isomorphism, the ``prequantisation map'' from $C^{\infty}(M)$, the observables of $(M,\omega)$, to the vector fields on $P$ that preserve $A$, i.e., the symmetries of $(P,A)$, is constructed:

\begin{equation*}
(K) \hspace{2cm}  f\mapsto \widetilde{X_f}+(q^*f)\cdot\mathfrak{z} \, ,  \hspace{3cm}
\end{equation*}

where, for $f\in C^{\infty}(M)$, $X_f$ denotes the Hamiltonian vector field fulfilling $df=\iota_{X_f}\omega$, $\widetilde{X_f}$ its 
horizontal lift to $P$ and $\mathfrak{z}$ the fundamental vector field of the $S^1$-action on $P$, defined by 
$\mathfrak{z}_p=\frac{d}{d\tau}\big\vert_{\tau=0} \, p\cdot \exp(i2\pi \, \tau)$. (See the fundamental articles of Kostant
and Souriau:  \cite{MR0294568} resp. \cite{MR207332}.) \\

The last step of geometric prequantisation is the linear representation of the above symmetries on the vector space of $S^1$-equivariant 
$\mathbb{C}$-valued functions on $P$, whereas the fourth step concerns the passage to quantisation and involves more deliberate and 
delicate choices, such as polarisations and metaplectic structures (compare, e.g., \cite{MR1183739}). Though 
completely general uniqueness statements for the outcome of geometric quantisation are lacking (due essentially to the above 
mentioned choices in the fourth step), crucial classes of symplectic manifolds, as cotangent bundles, K\"ahler manifolds and coadjoint 
orbits of several types of Lie groups are ``quantised'' without ambiguities. Moreover, geometric quantisation is widely accepted in 
mathematical and theoretical physics as a very useful source of intuition when it comes to quantisation of point particles.\\

In contrast, modelling and, especially, quantising classical field theories beyond point particles, i.e., with sources of dimension bigger 
than one, is a highly active research area. Progress here has most certainly a strong impact on mathematics and physics. 
The last years have seen a new wave of interest in formulating classical field theories in terms of multisymplectic geometry, i.e., via
manifolds provided with non-generate closed $(n{+}1)$-forms for $n\geq 1$ replacing the symplectic (a.k.a. 1-plectic) manifolds used 
in mechanics. (Such couples $(M,\omega)$ are also called ``$n$-plectic manifolds''; compare \cite{MR3937870} for a recent discussion
of their properties.) This approach to classical field theory was strongly advocated and developped by the Polish school around 
Kijowski and Tulczyjew in the seventies of the last century (see, e.g., \cite{MR334772}), and later by the Spanish school (see, e.g., 
\cite{MR2559661}), but can be traced back at least to work of Volterra in the 19th century. One main 
reason for the revival of multisymplectic geometry is certainly the increased acceptance of higher geometric and algebraic structures,
as the Lie $\infty$-algebra of observables on a multisymplectic manifold, formulated in its modern form by Stasheff and co-workers,
and Baez and Rogers (see \cite{MR2982023} for a very general construction of observable algebras).\\
 
 In this article, we develop the geometric part of prequantisation for 2-plectic manifolds.\\
 
Our first important result is the existence of a global geometrisation of the datum of an integral pre-2-plectic form $\omega$
on a manifold $M$. More precisely we prove:\\

{\bf Theorem (Global two-plectic geometric prequantisation).}
\textit{Let $\omega$ be an integral, closed three-form on a finite-dimensional manifold $M$. Then there exists a smooth
principal $PU(H)$-bundle $\Pi: Y\rightarrow M$,  such that the associated lifting bundle gerbe $(P,Y)$ has a connective structure 
$(A,\theta)$ whose 3-curvature equals $(-i2\pi)\omega$, and a fortiori the bundle $\Pi$ has Dixmier-Douady class $[\omega]$.}\\

Since the observables of a 2-plectic manifolds $(M,\omega)$ do not form a Lie algebra but a Lie 2-algebra  $L_{\bullet}(M,\omega)$,
given by the two-term complex $C^{\infty}(M) \stackrel{d}{\longrightarrow} \Omega^1_{Ham}(M,\omega)$ together with brackets $l_2$ and $l_3$,
the natural analogon of (K) is a Lie 2-algebra morphism whose target are ``vector fields on the total space $P$ of the gerbe (or
equivalently on the associated Lie groupoid $P\rightrightarrows Y$)''. We construct a natural Lie 2-algebra $\hbox{WSym}(P,Y)$ 
of weak symmetries out of the multiplicative vector fields on $P$ that preserve the chosen connective structure up to a one-form on $Y$, and 
$\mathfrak{X}_V(Y)\oplus C^{\infty}(Y)\cong \Gamma(Y, \hbox{Lie}(P))$, the sections of the Lie algebroid of $P$. Then we show that there is 
a natural, explicit Lie 2-algebra morphism from the observables to the symmetries, in close analogy to (K):\\

{\bf Theorem (Two-plectic prequantisation map).}
\textit{Let $(M,\omega)$ be a 2-plectic manifold with $\omega$ an integral form. Then there is a natural Lie 2-algebra 
quasi-isomorphism $\Phi=(\Phi_1,\Phi_2): L_{\bullet}(M,\omega)\rightarrow \hbox{WSym}(P,Y)$:}

\begin{center}
\begin{tikzcd}
C^{\infty}(M)\arrow[r] \arrow[d,"\Phi_1" '] 
& \Omega^1_{Ham}(M,\omega) \arrow[dl,tail,"\Phi_2"] \arrow[d,"\Phi_1"] \\
\mathfrak{X}_V(Y)\oplus C^{\infty}(Y) \arrow[r]
& \mathscr{V}_{(P,Y)} \hspace{0.3cm}.
\end{tikzcd}
\end{center}

\textit{Here  $\Omega^1_{Ham}(M,\omega)$ denotes the space of one-forms $\alpha$ defining a Hamiltonian vector field $X_{\alpha}$ via 
$d\alpha = \iota_{X_{\alpha}}\omega$, $L_{\bullet}(M,\omega)$ the Lie 2-algebra of observables of $(M,\omega)$, 
$\mathfrak{X}_V(Y)$ the vertical vector fields on the total space of the bundle  
$\Pi: Y\rightarrow M$ of the preceding theorem, and $\mathscr{V}_{(P,Y)}$ the degree-$0$ term of $\hbox{WSym}(P,Y)$.}\\

Let us briefly describe the organisation of this article. In the first section, we collect some ``folklore'' about principal bundles, lifting 
bundle gerbes and central $S^1$-extensions of Lie groupoids, together with the appropriate type of connections and connective 
structures. In case we could not point
to a proof in the literature, we furnished it for completeness. The main contribution is here the above theorem on 
global two-plectic geometric prequantisation. Section 2 relies on the dictionary from lifting bundle gerbes $(P,Y)$ to
a certain type of Lie groupoids $P\rightrightarrows Y$, namely central $S^1$-extensions of $Y^{[2]}\rightrightarrows Y$, where 
 $Y^{[2]}=\{(y_1.y_2) \in Y^2 \, \vert \, \Pi(y_1)=\Pi(y_2)  \}$ and $\Pi: Y\to M$ is the $PU(H)$-bundle constructed in Section 1. We analyse the 
 multiplicative vector fields on $P\rightrightarrows Y$ in thorough detail, explaining notably their $S^1$-invariance and their relation to objects on 
 the base $Y$. Then we explain the Lie 2-algebra structure on $\hbox{WSym}(P,Y)$, the above mentioned  weak symmetries of  
 $P\rightrightarrows Y$, together with a fixed multiplicative connection $A$ and a curving $\theta$. Finally, we unravel here the crucial
relation for multiplicative vector fields between weak preservation of the connective structure on $P$ and the existence 
of a Hamiltonian generator on $M$. In the third section, we give 
 the analogon of (K), namely the Lie 2-algebra quasi-isomorphism 
of the theorem on the two-plectic prequantisation map. To achieve this goal, we first observe that there is a natural Lie 2-algebra morphism from 
$L_{\bullet}(M,\omega)$ to the Lie 2-algebra formed by the sections of $\hbox{Lie}(P)$, identified with vertical vector fields plus functions on $Y$, 
and the space of all multiplicative vector fields on $P$, and then refine it to the morphism $\Phi$ of the theorem.\\

For the convenience of the reader, Appendix A collects basic definitions and facts from the theory of Lie 2-algebras, whereas in Appendix B we expose our constructions and results in detail in the case of exact 2-plectic manifolds, i.e., for manifolds $M$ with a non-degenerate 3-form $\omega$ possessing
a potential $\chi \in \Omega^2(M)$ such that $d\chi = \omega$. (The exact case is important for the application of our constructions 
to many of the classical field theories considered in physics.)\\

The results and proofs of this article will appear in part in the PhD thesis of the first named author. They are independent and different
from the research of Krepski and Vaughan in \cite{2020arXiv200312874K}, and of Fiorenza, Rogers and Schreiber (see, e.g., 
\cite{MR3241135}). The main difference stems from our objective to always give a global differential-geometric picture 
(including a $PU(H)$-bundle geometrising the cohomology class of the 2-plectic form) and explicit formulae, notably for the
arising Lie 2-algebra morphisms. For the reader interested in comparing our work to these approaches to prequantisation of
2-plectic manifolds, we insert precise remarks at several places in the body of the article.\\

Another approach to 2-plectic prequantization is given by Rogers in \cite{Rogers13} (compare also \cite{MR2912195}), where the Lie 2-algebra of 
observables of a 2-plectic manifold $(M,\omega)$ is embedded into a Lie 2-algebra associated to the sections of a Courant algebroid, that
in turn is associated to $(M,\omega)$. This procedure parallels the interpretation of the symplectic prequantization map (K) above as
realising the symplectic observables, $C^{\infty}(M)$, inside the sections of the Atiyah algebroid associated to an integral symplectic
manifold. This procedure differs from ours since we stress the construction of a global geometric prequantization, 
and the realisation of the observables as ``vector fields on this global geometric object'',
refined to weak infinitesimal symmetries.  Nevertheless, combining Rogers' above ideas with unpublished work of 
Collier (\cite{col}) should allow for a construction ana\-logous to our two-plectic prequantization map, but in Brylinski's more algebraic 
(\v{C}ech) framework.\\

Finally, a natural idea is to use transgression of differential forms to pass from a finite-dimensional integral 
2-plectic manifold $(M,\omega)$ to the induced integral two-form $\hat{\omega}$ on $LM$, the loop space of $M$, 
and to try quantising the latter (infinite-dimensional) (pre-)symplectic manifold. Prominent examples are here the standard Cartan 3-form 
on a compact semi-simple Lie group, leading to a derivation of many properties of Wess-Zumino-Witten modeles 
(compare, e.g., \cite{Kohno}), and volume forms on compact three-manifolds, leading to interesting geometry on knot spaces 
(cf. \cite{Brylinski}). More recently, S\"amann, Szabo and collaborators combined the idea of transgression with
symplectic groupoids and higher algebraic structure on the observables of a 2-plectic manifold to analyse properties of 
a tentative quantisation of integral 2-plectic manifolds (see, e.g., \cite{Saemann-Szabo}). We dot not use transgression to 
loop spaces here since we are aiming to stay in the 2-plectic framework. \\ 

(Pre-)quantization of manifolds with integral three-forms is also surveyed in a recent preprint of Bunk \cite{Bunk21}.\\

\noindent {\it Acknowledgements.}  We thank Camille Laurent-Gengoux for useful discussions related to the 
content of this article, notably on multiplicative tensors on Lie groupoids. We also wish to thank Jouko Mickelsson for a 
helpful remark concerning principal connections in infinite dimensions. The research of Gabriel Sevestre was financially supported
by the R\'egion Grand Est in France.

%
%

\section{Global geometric prequantisation of integral pre-2-plectic manifolds}

In this section we construct a global geometric prequantisation of a pre-2-plectic manifold. Our approach naturally relies on
 principal $PU(H)$-bundles, bundle gerbes and central extensions of groupoids. As a preparation, we show several technical results needed to ensure 
 the existence of connections on certain principal bundles in infinite dimensions. With the exception of Theorem \ref{GGPQ} all results are  
 explicitely known or fokloristic. We give proofs only in case we could not find a reference covering the exact statement.

\begin{defa}
Let $M$ be a manifold and $\omega\in \Omega^3(M)$. We say that $(M,\omega)$ is a {\it pre-2-plectic manifold} if $\omega$ is closed. 
\end{defa}

It is well-known that $H^3(M,\mathbb{Z})$ for a finite-dimensional manifold $M$, or for a suitable infinite-dimensional manifold, parametrizes the equivalence classes of smooth principal $PU(H)$-bundles over $M$ (here $H$ is a separable complex 
Hilbert space and $PU(H)$ the Banach Lie group of projective unitary isomorphisms of $H$). Thus, given an integral, closed 3-form $\omega$ on $M$, we have via the canonical surjection $H^3(M,\mathbb{Z})\rightarrow H^3_{dR}(M,\mathbb{Z})$ a class $c\in H^3(M,\mathbb{Z})$ and an associated principal $PU(H)$-bundle $Y\rightarrow M$ with characteristic class equal to $c$. The class $c$ is called the {\it Dixmier-Douady class} of $Y\rightarrow M$ (cf., e.g., \cite{MR2377868}), and projects to $[\omega]$. Non-unicity arises from torsion in $H^3(M,\mathbb{Z})$, exactly as in the case of geometric prequantisation of symplectic manifolds. Our first technical goal is to assure us of the existence of principal connections on these bundles. 
We treat a slightly more general situation than strictly needed here.

\begin{defa}
A possibly infinite-dimensional manifold $M$ is called \textit{smoothly paracompact} if for every open covering $\mathcal{U}=\{U_\alpha\}_{\alpha\in\Lambda}$ of $M$ there exists a smooth partition of unity $\{\chi_\alpha\}_{\alpha\in\Lambda}$ subordinate to $\mathcal{U}$.
\end{defa}

\begin{rem}
The preceding notion is studied in detail in Section III.16 of \cite{MR1471480}.
\end{rem}

\begin{prop}
Let $G$ be a Banach Lie group with Lie algebra $\mathfrak{g}$, $M$ a 
smoothly paracompact manifold and $\pi:Q\rightarrow M$ a principal $G$-bundle. Then there 
exists a smooth principal connection on $Q$.
\end{prop}

\begin{proof}
Let $\{U_\alpha\}_{\alpha\in\Lambda}$ be an open covering of $M$ such that there are principal $G$-bundle isomorphisms $\psi_\alpha:Q|_{U_\alpha}=\pi^{-1}(U_\alpha)\rightarrow U_\alpha \times G$ over $\hbox{id}_{U_\alpha}$, and let $A_\alpha$ be the pullback via $\psi_\alpha$ of the connection 1-form of the canonical flat connection on $U_\alpha\times G$. A direct verification shows (as in the finite-dimensional case) that the $\mathfrak{g}$-valued 1-form $A=\sum_\alpha (\chi_\alpha \circ\pi)\cdot A_\alpha$ is a connection 1-form on the $G$-bundle $\pi$.
\end{proof}

\begin{cor}
Let $M$ be a finite-dimensional manifold and $\omega$ a closed, integral 3-form on $M$. Then there exists a smooth principal $PU(H)$-bundle (with $H$ a separable complex Hilbert space), such that its Dixmier-Douady class equals (up to torsion) the class of $\omega$ in $H^3(M,\mathbb{Z})$, and all such bundles possess smooth principal connections.
\end{cor}

Before going to the gerbe intepretation of 3-forms, we need some more information concerning the existence of connections. We start with the following:

\begin{lem}
Let $H$ be a separable complex Hilbert space, $U(H)$ the unitary group of $H$ with the norm topology and $PU(H)=U(H)/S^1$ the projective unitary group of $H$ with the ensueing quotient topology. Then:

\begin{enumerate}[label=(\roman*)]
\item $S^1\rightarrow U(H)\rightarrow PU(H)$ is a short exact sequence of smooth Banach Lie groups,
\item the preceding sequence induces a smooth principal $S^1$-bundle $U(H)\rightarrow PU(H)$,
\item the principal $S^1$-bundle $U(H)\rightarrow PU(H)$ possesses a smooth principal connection.
\end{enumerate}
\end{lem}

\begin{proof}
(i) is well-known.\\

Let us recall that the Lie algebra of $U(H)$ is given as $\mathfrak{u}(H)=T_{\hbox{id}_H}U(H)=\{\xi\in \mathcal{B}(H) \ | \ \xi^*=-\xi\}$, where $\mathcal{B}(H)$ is the space of bounded linear endomorphisms of $H$. We denote the Lie algebra of $S^1=
\{\exp(it) \cdot \hbox{id}_H \, \vert \, t\in \mathbb{R} \}$ by $\mathfrak{z}=\{it\cdot \hbox{id}_H \ | \ t\in\mathbb{R}\}$ and of 
$PU(H)$ by $\mathfrak{pu}(H)=\mathfrak{u}(H)/\mathfrak{z}$. 

(ii) Observing that $\mathfrak{z}\subset \mathfrak{u}(H)$ is 1-dimensional, we construct a closed complement of this subspace
as follows. First let us fix the convention that the scalar product $<,>$ of $H$ is complex-linear in its second entry. 
Upon choosing an element $v$ of norm one in $H$, we define a continuous linear functional $\varphi$ on $\mathcal{B}(H)$ by 
$\varphi(\xi)=$ $-i <v,\xi(v)>$. It is clear that $\varphi$ takes real values on $\mathfrak{u}(H)$ and that 
$\varphi(i \cdot \hbox{id}_H)=1$.
The kernel of $\varphi$ will be denoted by $\mathfrak{m}$ and it is a closed linear complement of $\mathfrak{z}$, i.e.,
$$(*)  \hspace{1cm} \mathfrak{u}(H)=\mathfrak{z}\oplus\mathfrak{m} \hspace{1cm} \hbox{as real Banach spaces. }$$

Since the exponential map $\exp^{U(H)}:\mathfrak{u}(H)\rightarrow U(H)$ is given by the exponential series,
$\exp^{U(H)}$ is a local diffeomorphism and it follows as in the finite-dimensional case that $U(H)\rightarrow PU(H)$ is a smooth principal $S^1$-bundle (notably being locally trivial!).

(iii) Since $S^1$ is a central subgroup of $U(H)$, the splitting $(*)$  is invariant under the adjoint action of $U(H)$ restricted to $S^1$. Thus the canonical $\mathfrak{u}(H)$-valued 1-form on $U(H)$, followed by the projection onto $\mathfrak{z}$ defined by this splitting, yields a connection 1-form on $U(H)\rightarrow PU(H)$, that is furthermore invariant under the action of $U(H)$ on itself by left multiplication. (Compare, e.g., Theorem 11.1 in Chapter II of 
\cite{MR1393940} for a proof in a finite-dimensional situation, that immediately generalizes to our setup.)
\end{proof}

\begin{rem}
In fact, all maps and constructions in the preceding lemma are real-analytic.
\end{rem}

\begin{defa}
Let $\pi: Y\to M$ be a surjective submersion of smooth Banach manifolds. Then we set $Y^{[0]}:=M$, $Y^{[1]}:=Y$
and for a natural number $k\geq 2$, $Y^{[k]}:=Y\times_\pi Y \times_\pi ...  \times_\pi Y=\{(y_1,...,y_k)\in Y \ | \ \pi(y_1)=...=\pi(y_k)\}$. We call $Y^{[k]}$
the \textit{$k$-fold fiber product of $Y$ with itself (over $M$)}.
\end{defa}

\begin{rem}
All fiber products $Y^{[k]}$ for $k\geq 0$ are smooth since $\pi$ is assumed to be a surjective submersion.
\end{rem}

\begin{lem}
Let $G$ be a Banach Lie group and $Y\xrightarrow{\pi} M$ a principal $G$-bundle in the category of smooth Banach manifolds. Then there is a unique, smooth map $\psi:Y^{[2]}\rightarrow G$ such that for all $(y,y')\in Y^{[2]}$, $y'=y\cdot \psi(y,y')$. 
\end{lem}

\begin{proof}
The pair $(y,y')\in Y^2$ is in $Y^{[2]}$ if and only if $\pi(y)=\pi(y')$, that is if and only if $y$ and $y'$ are in the same orbit of the free and transitive $G$-action on $\pi$-fibers. Thus it exists a unique $g=\psi(y,y')\in G$ such that $y'=y\cdot g=y\cdot \psi(y,y')$, yielding $\psi: Y^{[2]}\rightarrow G$. The map $\psi$ is smooth by the implicit function theorem for differentiable maps between Banach manifolds.
\end{proof}

We now recall the definition of $S^1$-bundle gerbes that goes back to \cite{MR1405064}. 

\begin{defa}
Let $M$ be a smooth manifold. A \textit{$S^1$-bundle gerbe $(P,Y)$ over $M$} consists of:

\begin{itemize}
\item a smooth manifold $Y$ with a surjective submersion $Y\xrightarrow{\pi} M$,
\item a principal $S^1$-bundle $P$ over $Y^{[2]}$,
\item a smooth isomorphism $\mu:\pi_3^*P\otimes \pi_1^*P\rightarrow \pi_2^*P$ of principal $S^1$-bundles over $Y^{[3]}$ 
(where $\pi_i:Y^{[3]}\rightarrow Y^{[2]}$ denotes the projection that forgets the $i^{th}$ factor).
\end{itemize}

Furthermore, the isomorphism $\mu$ has to verify an associativity condition, which can be expressed by the commutativity of the following diagram of $S^1$-bundles over $Y^{[4]}$:

\begin{center}
\begin{tikzcd}
\pi_{34}^*P\otimes \pi_{14}^*P\otimes \pi_{12}^*P \arrow[r,"\mu\otimes id"] \arrow[d,"id\otimes\mu"]
& \pi_{24}^*P\otimes \pi_{12}^*P \arrow[d,"\mu"] \\
\pi_{34}^*P\otimes \pi_{13}^*P \arrow[r,"\mu"]
& \pi_{23}^*P   \, .
\end{tikzcd}
\end{center}

Here $\pi_{ij}:Y^{[4]}\rightarrow Y^{[2]}$ denotes the projection which forgets the $i^{th}$ and $j^{th}$ factors. 

\end{defa}

\begin{rem}
Since we are only dealing with $S^1$-bundle gerbes in this article, we allow ourselves to use the shorter term {\it bundle gerbes} for them
as well. 
\end{rem}

Suppose we are given a principal $G$-bundle $Y\rightarrow M$ with a Lie group  $G$ that possesses a central $S^1$-extension :

\begin{equation*}
1\rightarrow S^1 \rightarrow \widehat{G}\buildrel q\over\rightarrow G \rightarrow 1 \, ,
\end{equation*}

and consider the map $\psi:Y^{[2]}\rightarrow G$ (as in Lemma 1.10). Then we can pull back the principal $S^1$-bundle $\widehat{G}\rightarrow G$ via $\psi$ to obtain a principal $S^1$-bundle $P\rightarrow Y^{[2]}$. We may then define a map $\mu$ via the multiplication on $\widehat{G}$, namely:

\begin{equation*}
\mu((y,y',\hat{g}),(y',y'',\hat{g}'))=(y,y'',\hat{g}\cdot \hat{g}'),
\end{equation*}

for $y,y',y''\in Y^{[3]}$ and $\hat{g},\hat{g}'\in \widehat{G}$. This gives a bundle gerbe $(P,Y)$ over $M$, called the \textit{lifting bundle gerbe associated to} $Y\rightarrow M$.\\

There is an equivalent description of bundle gerbes, via extensions of groupoids. Since the below result will be crucial for us and it is only hinted at in the original source \cite{MR2817778}, we give the idea of its proof.

First we need the

\begin{defa}
Let $G_1\rightrightarrows G_0$ be a Lie groupoid. A Lie groupoid $Q_1\rightrightarrows G_0$ is called an 
$S^1${\it -central extension of}
 $G_1\rightrightarrows G_0$ if there exists an exact sequence of Lie groupoids

\begin{equation*}
1\rightarrow G_0\times S^1\buildrel \chi\over\rightarrow Q_1 \buildrel \rho\over\rightarrow G_1 \rightarrow 1
\end{equation*}

over $\hbox{id}_{G_0}$ such that the image of $\chi$ is central in $Q_1$.
\end{defa}

\begin{prop}
Given a surjective submersion $Y\xrightarrow{\pi}M$, the $S^1$-bundle gerbes $(P,Y)$ over $M$ (with the given $\pi$ part of the data) 
are in bijection with the central $S^1$-extensions of the Lie groupoid $Y^{[2]} \rightrightarrows Y$ 
(with $Y^{[2]}$ defined by the given $\pi$).

\end{prop}

\begin{proof}
Fix $\pi: Y \rightarrow M$ during this proof. \\

Given a bundle gerbe in the above sense, we have notably a principal $S^1$-bundle $P \rightarrow Y^{[2]}$. Denoting its projection by $\rho$, we 
have the second morphism of a sequence

\begin{equation*}
(*) \hskip1cm 1\rightarrow Y\times S^1\rightarrow P \rightarrow Y^{[2]} \rightarrow 1. 
\end{equation*}

Furthermore, if $\Delta: Y\rightarrow  Y^{[2]}$ is the diagonal embedding, we have a canonical section of the pullback bundle
$\Delta^*P\rightarrow Y$ and an ``inverse'' $P_{(y_1,y_2)} \rightarrow P_{(y_2,y_1)}, p\mapsto p^{-1}$ (cf. \cite{MR1405064}
for these facts). The map $\rho$ now yields a source $s=\pi_2 \circ \rho$ and a target $t=\pi_1 \circ \rho$, as well as a multiplication 
map $m: P {{}_s\times_t} P\rightarrow P$, defined in terms of $\mu$, and a unit map 
$\varepsilon: Y\rightarrow P, y\mapsto \varepsilon(y)\in P_{(y,y)}$. Obviously 
$\rho$ is then a surjective Lie groupoid morphism. The kernel of $\rho$ is easily seen to be the image of the injective morphism 
$\chi: Y\times S^1 \rightarrow P, \chi(y,z)=\varepsilon(y)\cdot z$.\\

On the other hand, given a central $S^1$-extension as described by $(*)$, we define a right-action of $S^1$ on $P$ by $p\cdot z := m(p,\chi(y,z))$
for $z\in S^1$ and $p$ in $P$ lying over $(y',y)$ in $Y^{[2]}$, and $m$ the multiplication of the groupoid $P\rightrightarrows Y$. Finally, we define
$\mu$ via the groupoid multiplication $m: P_{(y_1,y_2)}\times  P_{(y_2,y_3)} \rightarrow  P_{(y_1,y_3)}$. The associativity of the multiplication $m$ 
implies the associativity condition necessary for $\mu$ to be a ``bundle gerbe multiplication''. 
\end{proof}

We now briefly recall connective structures on bundle gerbes whose detailed study can be found in the literature (see, e.g., 
\cite{MR1405064}, \cite{stevphd}). A \textit{bundle gerbe connection} on a bundle gerbe $(P,Y)$ over $M$ is given by a connection one-form
$A$ on the total space of the principal $S^1$-bundle $P\rightarrow Y^{[2]}$, such that $A$ is \textit{multiplicative} in the following sense:

\begin{equation*}
m^*A=\hbox{proj}_1^*A+\hbox{proj}_2^*A   \, ,
\end{equation*}

where $m,\hbox{proj}_1, \hbox{proj}_2:P {{}_s\times_t} P\rightarrow P$ are respectively the partial multiplication and the projections onto the first and second factor
(the maps $s$ and $t$ denote the source and target maps of the groupoid). In fact, given any connection on the principal $S^1$-bundle $P\rightarrow Y^{[2]}$, it is always possible to modify it to obtain a multiplicative connection. For this and the sequel up to the theorem, compare Section 4.4 of \cite{MR2681698}.\\

The curvature of such a connection descends to a 2-form $\theta$ on $Y$, more precisely, there exists a $\theta\in i \Omega^2(Y)$ fulfilling:

\begin{equation*}
dA=(t^*-s^*)\theta.
\end{equation*}

We call such a 2-form $\theta$ a \textit{curving}. Finally, since $d(t^*-s^*)\theta=0$, there exists $\kappa\in i\Omega^3(M)$ such that $d\theta=\Pi^*\kappa$. 
(We denote from now the map $Y\to M$ by $\Pi$ to distinguish it from the number $\pi$.) We call $\kappa$ the \textit{3-curvature} and observe that $(-2\pi i) \kappa$ is integral.
Note that the curving $\theta$ is not unique, in fact, if $\theta'=\theta+\Pi^*\eta$, where $\eta$ is a closed, $i\mathbb{R}$-valued 
2-form on $M$, we still have 
$(t^*-s^*)\theta'=dA$ and $d\theta'=\Pi^*\kappa$. Finally, $\frac{i}{2\pi} \kappa$ represents the image of the Dixmier-Douady class $c$ of the
bundle gerbe $(P,Y)$ over $M$ in the de Rham cohomology of $M$.

In the sequel, we will often call $(A,\theta)$ a {\it connective structure on} $(P,Y)$. \\

Now consider a class $c\in H^3(M,\mathbb{Z})$, and $Y\rightarrow M$ a principal $PU(H)$-bundle with Dixmier-Douady class equal to $c$. Then the lifting bundle gerbe associated to the exact sequence $S^1\rightarrow U(H) \rightarrow PU(H)$ also has Dixmier-Douady class $c$ (see e.g. \cite{MR2681698}). This will be useful in the proof of the next theorem, which will allow us to define a global geometric prequantisation of a pre-2-plectic manifold.\\

\begin{thm}\label{GGPQ}
Let $\omega$ be an integral, closed 3-form on a finite-dimensional manifold $M$. 
Then there exists a smooth principal $PU(H)$-bundle $\Pi: Y\rightarrow M$, 
such that the associated lifting bundle gerbe $(P,Y)$ has a connective structure $(A,\theta)$ whose 3-curvature equals
$(-2\pi i) \omega$ and, a fortiori, the bundle $\Pi$ has Dixmier-Douady class $[\omega]$.
\end{thm}

\begin{proof}
We know that there exists a principal $PU(H)$-bundle $\Pi: Y\rightarrow M$ with 
Dixmier-Douady class projecting to $[\omega]$. Consider the lifting bundle gerbe $(P,Y)$ associated to this bundle. We know from 
Lemma 1.6 that 
there exists a principal connection on $U(H)\rightarrow PU(H)$, which we pullback to $P\rightarrow Y^{[2]}$ via $\psi$. (Recall that the principal bundle $P\rightarrow Y^{[2]}$ can be realized as the pullback of  $U(H)\rightarrow PU(H)$  via $\psi$.) If necessary, we modify this connection and we can 
assume it to be multiplicative. Let us denote the resulting connection 1-form by $A$. Then the above considerations imply
that there exists $\theta'\in\Omega^2(Y)$ such that $(t^*-s^*)\theta'=dA$ and $\omega'\in\Omega^3(M)$ such that $d\theta'=\Pi^*\omega'$. Since 
$[\omega-\omega']=0$ there exists a 2-form $\kappa\in\Omega^2(M)$ such that $\omega=\omega'+d\kappa$. Consider $\theta=\theta'+\Pi^*\kappa$. Replacing $\theta'$ by $\theta$ in the connective structure of the lifting bundle gerbe does not change the connection of the principal $S^1$-bundle over $Y^{[2]}$, and now we have $d\theta=d\theta'+\Pi^*d\kappa=\Pi^*(\omega'+d\kappa)=\Pi^*(\omega)$. 
\end{proof}

\begin{defa}
A \textit{global geometric prequantisation} of an integral pre-2-plectic manifold $(M,\omega)$ is given by a principal $PU(H)$-bundle
$\Pi: Y\rightarrow M$, and a bundle gerbe $(P,Y)$ over $M$ with 
connective structure $(A,\theta)$ whose 3-curvature is equal to $(-2\pi i) \omega$. 
\end{defa}

\begin{rem}
In Definition 2.2.4 of \cite{MR3241135} a {\it prequantisation} is essentially a connective structure with 3-curvature equal
to  $(-2\pi i) \omega$. Our definition asks for a more complete set-up for
the connection to arise, notably a principal $PU(H)$-bundle and a 
gerbe (or equivalently, a central $S^1$-extension of groupoids) on its total space having the ``right'' 3-curvature. 
In \cite{2020arXiv200312874K}, the notion of a quantisation 
used is closer to ours,
but the analogon of our $PU(H)$-bundle projection $\Pi$ is there only a surjective submersion.

\end{rem}

%
%

\section{Infinitesimal symmetries of bundle gerbes}

In this section we turn our attention to \textit{multiplicative} vector fields on groupoids, intensely studied by Kirill Mackenzie
(see \cite{MR1617335} or \cite{MR4023383} for a more complete treatment of multiplicative tensors on Lie groupoids). 

\begin{defa}
Let $G_1\rightrightarrows G_0$ be a Lie groupoid. A \textit{multiplicative vector field (on 
$G_1\rightrightarrows G_0$)} is given by a couple
$(\xi,X)$  in  $\mathfrak{X}(G_1)\oplus \mathfrak{X}(G_0)$ such that the following diagram is a Lie groupoid morphism:

\begin{center}
\begin{tikzcd}
G_1  \arrow[r,"\xi"] \arrow[d, shift right=0.7ex]   \arrow[d, shift left=0.6ex]
&  TG_1  \arrow[d, shift right=0.7ex] \arrow[d, shift left=0.6ex]
\\
G_0 \arrow[r,"X"] 
&  TG_0    
\end{tikzcd}  
\end{center}

The Lie algebra of multiplicative vector fields is denoted by $\mathfrak{X}_{mult}(G_1)$.

\end{defa}

For the rest of the section, we will denote by $M$ a manifold, by $(P,Y)$  a bundle gerbe over $M$, by $A$ a multiplicative 
connection on $P\rightarrow Y^{[2]}$, and by $\theta$ resp. $\omega$ a curving resp.
the 3-curvature of $(A,\theta)$.  As discussed earlier, we have a natural Lie groupoid structure $P\rightrightarrows Y$, centrally 
extending $Y^{[2]}\rightrightarrows Y$ by $S^1$. We denote by $s,t$ the source resp. target maps of this groupoid, $m$ 
its partial multiplication and $\varepsilon$ its unit map. We denote the principal $S^1$-bundle projection 
$P\rightarrow Y^{[2]}$ again by $\rho$. Observe that $\rho=(t,s)$.\\

We will consider multiplicative vector fields on the Lie groupoid $P\rightrightarrows Y$ in more detail now, starting 
with two important lemmas.

\begin{lem}
Let $(\xi,X)$ be a multiplicative vector field on the Lie groupoid $P\rightrightarrows Y$. Then $\xi$ is $S^1$-invariant.
\end{lem}

\begin{proof}
In this proof, we identify the Lie algebra of $S^1$ with $\mathbb{R}$. We begin with a computation, for $y\in Y$, $z,z'\in  S^1$:
\begin{equation*}
\begin{split}
A_{\varepsilon(y)\cdot zz'}(\xi_{\varepsilon(y)\cdot zz'})&=A_{m(\varepsilon(y)\cdot z,\varepsilon(y)\cdot z')}(\xi_{m(\varepsilon(y)\cdot z,\varepsilon(y)\cdot z')}) \\
&=A_{m(\varepsilon(y)\cdot z,\varepsilon(y)\cdot z')}(m_{*}(\xi_{\varepsilon(y)\cdot z},\xi_{\varepsilon(y)\cdot z'})) \\
&=A_{\varepsilon(y)\cdot z}(\xi_{\varepsilon(y)\cdot z})+A_{\varepsilon(y)\cdot z'}(\xi_{\varepsilon(y)\cdot z'}).
\end{split}
\end{equation*}
Thus we obtain, for a fixed $y\in Y$, a smooth map $F:S^1\rightarrow \mathbb{R}$, $F(z)=A_{\varepsilon(y)\cdot z}(\xi_{\varepsilon(y)\cdot z})$, such that $F(zz')=F(z)+F(z')$. This shows that $F$ must be zero, and therefore that for every $y\in Y$, $z\in S^1$, $\xi_{\varepsilon(y)\cdot z}$ is horizontal. Since the horizontal subbundle $H\subset TP$ is invariant under $S^1$, this also shows that there exists a horizontal vector $u(y,z)\in H_{\varepsilon(y)}$ such that $\xi_{\varepsilon(y)\cdot z}=z_{*}(u(y,z))$ (where $z$ also denotes the right-action of an element $z\in S^1$). For $z=1$ we obtain $u(y,1)=\xi_{\varepsilon(y)}=\varepsilon_{*}(X_y)$. Moreover we have, using the multiplicativity of $\xi$:

\begin{equation*}
\begin{split}
\rho_{*}(u(y,z))&=\rho_{*}(\xi_{\varepsilon(y)\cdot z}) \\
&=(s,t)_{*}(\xi_{\varepsilon(y)\cdot z}) \\
&=(X_y,X_y) \\
&=\rho_{*}(u(y,1)) \, .
\end{split}
\end{equation*}

So the vector $u(y,z){-}u(y,1)$ is both horizontal and vertical, and therefore must be zero. Thus we obtain that 
$\xi_{\varepsilon(y)\cdot z}=z_{*}(\xi_{\varepsilon(y)})$.\\

Finally we compute, for $(y,y')\in Y^{[2]}$, $p\in P_{(y,y')}$ and $z\in S^1$ (using again the multiplicativity of $\xi$):

\begin{equation*}
\begin{split}
\xi_{p\cdot z}&=\xi_{m(p,\varepsilon(y')\cdot z)}  \\
&=m_{*}(\xi_p,\xi_{\varepsilon(y')\cdot z}) \\
&=m_{*}(\xi_p,z_{*}\xi_{\varepsilon(y')}) \\
&=z_{*}(m_{*}(\xi_p,\xi_{\varepsilon(y')})) \\
&=z_{*}\xi_p.
\end{split}
\end{equation*}

This concludes the proof.
\end{proof}

\begin{rem}
Observe that, if $(\xi,X)$ is a multiplicative vector field on $P\rightrightarrows Y$, then we have $s_*\xi=X\circ s$, $t_*\xi=X\circ t$, and, 
since $\Pi\circ s=\Pi\circ t$, we obtain for all $(y,y')\in Y^{[2]}$:
\begin{equation*}
\Pi_{*y}X_y=\Pi_{*y'}X_{y'}.
\end{equation*}
This shows that the vector field $X$ is \textit{projectable along $\Pi$}, i.e.,  there exists $\bar{X}\in \mathfrak{X}(M)$ such that 
$\Pi_*X=\bar{X}\circ\Pi$. Thus, for $(y,y')\in Y^{[2]}$ the map:

\begin{equation*}
(y,y')\mapsto (X_y,X_{y'})
\end{equation*}

defines a vector field on $Y^{[2]}$. We will denote this vector field by $X_2$.\\
\end{rem}

These considerations allow us to give a general formula for multiplicative vector fields on $P\rightrightarrows Y$.

\begin{lem}\label{multiplicative-vf-on-P}
Let $(\xi,X)$ be a multiplicative vector field on the Lie groupoid $P\rightrightarrows Y$. Then $\xi$ may be written as:

\begin{equation*}
\widetilde{X_2} + (t^*-s^*)g\cdot\mathfrak{z},
\end{equation*}

where $\widetilde{X_2}$ is the horizontal lift of the vector field $X_2$ with respect to the multi\-plicative connection $A$, 
$g\in C^{\infty}(Y)$ and $\mathfrak{z}$ is the fundamental vector field associated to the $S^1$-action on $P$. The function $g$ 
is unique up to addition of  $\Pi^*f$ for a $f\in C^{\infty}(M)$.
\end{lem}

\begin{proof}
Since $\xi$ is $S^1$-invariant, it can be written as
\begin{equation*}
\widetilde{Z}+\rho^*h\cdot\mathfrak{z} \, ,
\end{equation*}
where $Z\in\mathfrak{X}(Y^{[2]})$, $\widetilde{Z}$ is its horizontal lift with respect to $A$ and $h\in C^{\infty}(Y^{[2]})$. We have for $p\in P_{(y,y')}$ 
\begin{equation*}
\rho_{*p}(\widetilde{Z}_p+(\rho^*h\cdot\mathfrak{z})_p)=Z_{(y,y')}
\end{equation*}
because $\widetilde{Z}$ is $\rho$-related to $Z$. On the other hand
\begin{equation*}
\rho_{*p}(\widetilde{Z}_p+(\rho^*h\cdot\mathfrak{z})_p)=(t,s)_{*p}(\widetilde{Z}_p)=(X_y,X_{y'})
\end{equation*}
because $(\xi, X)=(\widetilde{Z}+\rho^*h\cdot\mathfrak{z},X)$ is multiplicative. This shows that $Z=X_2$. Thus $(\widetilde{X_2},X)$ is multiplicative, and 
the vector field $(\rho^*h\cdot\mathfrak{z},0)$ must be multiplicative as well.\\

We will now show the following identity for all $y,y',y''\in Y^{[3]}$:

\begin{equation*}
h(y,y')+h(y',y'')=h(y,y')
\end{equation*}

implying in turn that the function $h$ descends to a function on $Y$. We have, for all $p\in P_{(y,y')}$, $p'\in P_{(y',y'')}$

\begin{equation*}
\begin{split}
h(y,y')+h(y',y'')&=A_p((\rho^*h\cdot\mathfrak{z})_p)+A_{p'}((\rho^*h\cdot\mathfrak{z})_{p'}) \\
&=(m^*A)_{(p,p')}((\rho^*h\cdot\mathfrak{z})_p,(\rho^*h\cdot\mathfrak{z})_{p'}) \\
&=A_{m(p,p')}(m_{*(p,p')}((\rho^*h\cdot\mathfrak{z})_p,(\rho^*h\cdot\mathfrak{z})_{p'})) \\
&=A_{m(p,p')}((\rho^*h\cdot\mathfrak{z})_{m(p,p')}) \\
&=h(y,y'') \, ,
\end{split}
\end{equation*}

where we used the multiplicativity of $\rho^*h\cdot\mathfrak{z}$ to obtain the fourth equality. Thus the function $h\in C^{\infty}(Y^{[2]})$ descends to a function $g\in C^{\infty}(Y)$. Finally, since $(t^*-s^*)\circ \Pi^*=0$, the function $g\in C^{\infty}(Y)$ is unique only up to addition of $\Pi^*f$ with $f$ a smooth function on $M$.
\end{proof}

In \cite{2016arXiv160903944B} the authors construct a strict Lie 2-algebra structure on multiplicative vector fields. In fact, this construction generalizes to any $\mathscr{LA}$-groupoid (\cite{MR3985597}). See the appendix for the relevant definitions and notations concerning Lie 2-algebras and crossed modules of Lie algebras, as well as this construction. 
We will here, however, describe the Lie 2-algebra of multiplicative vector fields more explicitely, in a fashion taylored to our needs.\\

Consider the extension of groupoids 

\begin{center}
\begin{tikzcd}
Y\times S^1 \arrow[r] \arrow[d, shift right=0.7ex]   \arrow[d, shift left=0.6ex]
& P \arrow [r] \arrow[d, shift right=0.7ex]   \arrow[d, shift left=0.6ex]
& Y^{[2]} \arrow[d, shift right=0.7ex]   \arrow[d, shift left=0.6ex] \\
Y \ar[equal]{r}
& Y \ar[equal]{r}
& Y
\end{tikzcd}  
\end{center}

We obtain an exact sequence of vector bundles over $Y$ 

\begin{center}
\begin{tikzcd}
Y\times \mathbb{R} \arrow[r] \arrow[d] 
& \hbox{Lie}(P) \arrow [r] \arrow[d]
& T_VY \arrow[d] \\
Y \ar[equal]{r}
& Y \ar[equal]{r}
& Y
\end{tikzcd}
\end{center}

where the vector  bundles considered here are the Lie algebroids of the respective Lie groupoids. Here we have denoted by
$\hbox{Lie}(P)=(\ker s_*)\vert_{\varepsilon(Y)}\subset TP\vert_{\varepsilon(Y)}$ the Lie algebroid of the Lie groupoid 
$P\rightrightarrows Y$, and $T_V Y:=\ker(\Pi_{*})$, where $\Pi$ is the projection from $Y$ to $M$. The sections of the 
vector bundle $T_V Y$ form the space of vertical vector fields,
denoted by $\mathfrak{X}_V(Y)$. The map $t_* : \hbox{Lie}(P)\to T_VY$ is the anchor of the Lie algebroid 
$\hbox{Lie}(P)$.\\

The multiplicative connection $A$ defines an explicit  
splitting $\sigma:\mathfrak{X}_V(Y)\rightarrow \Gamma(Y, \hbox{Lie}(P))$. First let $Z\in\mathfrak{X}_V(Y)$. Then the map:
\begin{equation*}
(y,y')\mapsto (Z_y,0)
\end{equation*}

for $(y,y')\in Y^{[2]}$, defines a vector field on $Y^{[2]}$, since $Z$ is vertical. We denote this vector field by $(Z,0)$. We also set 
$(0,Z):(y,y')\mapsto (0,Z_{y'})$. Then we define:

\begin{equation*}
\sigma(Z)(y):=\widetilde{(Z,0)}_{\varepsilon(y)}  \, .
\end{equation*}

This map $\sigma$ is indeed a splitting of the above sequence of vector bundles (but not of algebroids!) 
since $t_{*}(\sigma(Z)(y))=t_{*}(\widetilde{(Z,0)}_{\varepsilon(y)})=Z_y$. 
Thus we have an isomorphism of $C^{\infty}(Y)$-modules 
$\mathfrak{X}_V(Y)\oplus C^{\infty}(Y)\xrightarrow{}\Gamma(Y, \hbox{Lie}(P)), \, (Z,h)\mapsto \sigma(Z)+ \kappa(h)$, 
where $\kappa(h)(y)=h (y) \cdot\mathfrak{z}|_{\varepsilon(y)}$. This isomorphism allows us to describe the strict Lie 2-algebra structure on $\mathfrak{X}_V(Y)\oplus C^{\infty}(Y)\rightarrow \mathfrak{X}_{mult}(P)$ in useful detail.\\

\begin{rem}(Compare \cite{MR2157566} and \cite{MR1617335}). We recall several facts concerning Lie groupoids and algebroids, which will be useful in the following. Let $G_1\rightrightarrows G_0$ be a Lie groupoid and $\hbox{Lie}(G_1)\rightarrow G_0$ the associated Lie algebroid. 
\begin{enumerate}[label=(\arabic*)]
\item The sections of $\hbox{Lie}(G_1)$ are isomorphic as a $C^{\infty}(G_0)$-module to the right-invariant resp. left invariant vector fields 
on $G_1$ via standard identifications $a\mapsto \overrightarrow{a}$ resp. $a\mapsto \overleftarrow{a}$. 
\item The Lie bracket on $\Gamma(G_0,\hbox{Lie}(G_1))$ is inferred from the right-invariant vector fields on $G_1$.
\item The map $\Gamma(G_0,\hbox{Lie}(G_1))\rightarrow \mathfrak{X}_{mult}(G_1)$, $a\mapsto (\overrightarrow{a}+\overleftarrow{a},X(a))$, where $X(a)$ is the vector field on $G_0$ associated to $a$ via the anchor map of $\hbox{Lie}(G_1)$, is a Lie algebra morphism. 
\end{enumerate}
\end{rem}

We can now give explicit formulas for the crossed module $\Gamma(Y,\hbox{Lie}(P))\rightarrow \mathfrak{X}_{mult}(P)$.

\begin{prop}
Let $Y\xrightarrow{\Pi}M$, $P\xrightarrow{\rho} Y^{[2]}$ and $(A,\theta)$ be as throughout this section. Then:
\begin{enumerate}[label=(\roman*)]
\item $\mathfrak{X}_V(Y)\oplus C^{\infty}(Y)\rightarrow \Gamma(Y,\hbox{Lie}(P))$, $(Z,h)\mapsto \sigma(Z)+\kappa(h)$, is a $C^{\infty}(Y)$-module 
isomorphism,
\item the Lie bracket on $\Gamma(Y,\hbox{Lie}(P))$ induces the following bracket on $\mathfrak{X}_V(Y)\oplus C^{\infty}(Y)$: 
\newline $[(Z,h),(Z',h')]=([Z,Z'],Z(h')-Z'(h)-\theta(Z,Z'))$,
\item the Lie algebra morphism $\Gamma(Y,\hbox{Lie}(P))\rightarrow \mathfrak{X}_{mult}(P)$, 
$a\mapsto (\overrightarrow{a}+\overleftarrow{a}, X(a) )$ 
translates to $\mathfrak{X}_V(Y)\oplus C^{\infty}(Y) \xrightarrow{ \overset{\circ}{\eta}} \mathfrak{X}_{mult}(P)$, 
\newline $(Z,h)\mapsto (\widetilde{Z_2}+ (t^*-s^*)h\cdot\mathfrak{z},Z)$,
\item the action of $\mathfrak{X}_{mult}(P)$ on $\Gamma(Y,\hbox{Lie}(P))$ given by $(\xi,X).a=b$ with 
$\overrightarrow{b}=[\xi,\overrightarrow{a}]$ induces the following action on 
$\mathfrak{X}_V(Y)\oplus C^{\infty}(Y)$: \newline $(\xi,X).(Z,h)=([X,Z],X(h)-Z(g)-\theta(X,Z))$, where 
$\xi = \widetilde{X_2} + (t^*-s^*)g\cdot\mathfrak{z}$ as in Lemma \ref{multiplicative-vf-on-P}.
\end{enumerate}
\end{prop}

\begin{proof}

The first statement has already been proven. Let now $(Z,h)$ be in $\mathfrak{X}_V(Y)\oplus C^{\infty}(Y)$.

The vector field $\overrightarrow{\sigma(Z)}$ is horizontal and:

\begin{equation*}
\begin{split}
\rho_*\left(\overrightarrow{\sigma(Z)}\right)(p)&=\rho_*\circ (R_p)_{*}(\widetilde{(Z,0)}_{\varepsilon(t(p))}) \\
&=(t_*(\widetilde{(Z,0)}_{\varepsilon(t(p)}),0) \\
&=(Z_{t(p)},0) \, .
\end{split}
\end{equation*}

Thus $\overrightarrow{\sigma(Z)}=\widetilde{(Z,0)}$. We continue with the computation

\begin{equation*}
\begin{split}
\overrightarrow{\kappa(h)}(p)&=(R_p)_{*}\left(\frac{d}{d\tau}\Big\vert_{\tau=0} \, \varepsilon(t(p))\cdot \exp(i 2\pi\tau \, h(t(p)))\right) \\
&=\frac{d}{d\tau}\Big\vert_{\tau=0}(p\cdot \exp(i2\pi\tau \, h(t(p))) \\
&=(t^*h\cdot\mathfrak{z})|_p.
\end{split}
\end{equation*}

In view of the $\mathbb{R}$-linear isomorphism from part (i), the Lie bracket on $\mathfrak{X}_V(Y)\oplus C^{\infty}(Y)$ 
can be determined from the below commutator of right-invariant vector fields:

\begin{equation*}
\begin{split}
&[\overrightarrow{\sigma(Z)+ \kappa(h)},\overrightarrow{\sigma(Z')+ \kappa(h')}]\\
&=[\overrightarrow{\sigma(Z)},\overrightarrow{\sigma(Z')}] -[\overrightarrow{\sigma(Z')},\overrightarrow{ \kappa(h)}] + [\overrightarrow{\sigma(Z)},\overrightarrow{ \kappa(h')}] + [\overrightarrow{ \kappa(h)},\overrightarrow{ \kappa(h')}] \\
&=[\widetilde{(Z,0)},\widetilde{(Z',0)}] -[\widetilde{(Z',0)},t^*h\cdot\mathfrak{z}] +[\widetilde{(Z,0)},t^*h'\cdot\mathfrak{z}] \\
&=\widetilde{([Z,Z'],0)} -t^*\theta(Z,Z')\cdot\mathfrak{z} -t^*(Z'(h))\cdot\mathfrak{z} +t^*(Z(h'))\cdot\mathfrak{z} \\
&=\sigma([Z,Z']+ \kappa(Z(h')-Z'(h)-\theta(Z,Z')). \\
\end{split}
\end{equation*}

We deduce that the Lie bracket is given by \newline $[(Z,h),(Z',h')]=([Z,Z'],Z(h')-Z'(h)-\theta(Z,Z'))$.\\

Similarly to the above calculation, one obtains $\overleftarrow{\sigma(Z)}=\widetilde{(0,Z)}$, 
and thus $\overrightarrow{\sigma(Z)}+\overleftarrow{\sigma(Z)}=\widetilde{Z_2}$. 
Also,  $\overleftarrow{\kappa(h)}=-s^*h\cdot\mathfrak{z}$ and 
$\overrightarrow{\kappa(h)}+\overleftarrow{\kappa(h)}=(t^*-s^*)h\cdot\mathfrak{z}$. Therefore the Lie algebra morphism
 $\mathfrak{X}_V(Y)\oplus C^{\infty}(Y)\xrightarrow{\varphi} \mathfrak{X}_{mult}(P)$ is explicitely described by

\begin{equation*}
\overset{\circ}{\eta}(Z,h)=(\widetilde{Z_2}+(t^*-s^*)h\cdot\mathfrak{z},Z).
\end{equation*}

Finally, the action of a multiplicative vector field $(\xi,X)$ on $(Z,h)$ reads

\begin{equation*}
\begin{split}
[\xi,\overrightarrow{\sigma(Z)+ \kappa(h)}]&=[\widetilde{X_2}+(t^*-s^*)g\cdot\mathfrak{z},\widetilde{(Z,0)}+t^*h\cdot\mathfrak{z}] \\
&=\widetilde{([X,Z],0)}-t^*\theta(X,Z)\cdot\mathfrak{z}+t^*X(h)\cdot\mathfrak{z}-t^*Z(g)\cdot\mathfrak{z} \\
&=\sigma([X,Z])+ \kappa(X(h)-Z(g)-\theta(X,Z)). \\
\end{split}
\end{equation*}

Thus the action is given by $[(\xi,X),(Z,h)]=([X,Z],X(h)-Z(g)-\theta(X,Z))$.\\

Since $Z$ is vertical, $Z(\Pi^*f)=0$ for all 
$f\in C^{\infty}(M)$. Accordingly, adding $\Pi^*f$ to $g$ does not change the right hand side in the preceding equation. 
With Lemma 2.4, it follows that the action in (iv) is well-defined. 
\end{proof}

Going back to the Lie 2-algebra structure, we have obtained the following

\begin{prop}\label{MVF=Lie2}
The Lie algebras $\mathfrak{X}_V(Y)\oplus C^{\infty}(Y) \left(\cong \Gamma(Y, \hbox{Lie}(P))\right) $ and $ $ $\mathfrak{X}_{mult}(P)$, together with 
the map $\overset{\circ}{\eta}$ and the action defined above, form a Lie algebra crossed module, hence a strict Lie 2-algebra. We will denote this 
Lie 2-algebra by $\mathfrak{X}_{mult}(P,\Pi):=\mathfrak{X}_{mult}(P,\Pi,A)$. 
\end{prop}

\begin{rem}
As mentioned above, the crossed module structure is abstractly known to exist on the complex 
$\Gamma(Y, \hbox{Lie}(P))\rightarrow \mathfrak{X}_{mult}(P)$
(cf. \cite{2016arXiv160903944B} and \cite{MR3985597}).
Here, we simply expressed the section space of the Lie algebroid and the bracket resp. the action of  $\mathfrak{X}_{mult}(P)$
on it in a convenient form. 
\end{rem}

\vspace{0.4cm}

In symplectic prequantisation, there is the important notion of vector fields that preserve the connection (``infinitesimal
quantomorphisms''). It turns out that the direct 
generalisation of this notion is too naive but the following definition  gives a very useful 2-plectic analogue.

\begin{defa}
Let $(\xi,X)$ be a multiplicative vector field on the Lie groupoid $P\rightrightarrows Y$ with connective structure $(A,\theta)$,
and $B$  a 1-form on $Y$.
We say that $(\xi,X,B)$ is a \textit{weak infinitesimal symmetry} if  

\begin{itemize}
\item $\mathscr{L}_{\xi} A= (t^*-s^*)B$
\item $ \mathscr{L}_X \theta =dB$ .
\end{itemize}
\end{defa}

We denote the space of weak infinitesimal symmetries by $\mathscr{V}_{(P,Y)}:=\mathscr{V}_{(P,Y;A,\theta)}:=
\{(\xi,X,B)\in \mathfrak{X}_{mult}(P)\oplus \Omega^1(Y)\ | \, \mathscr{L}_{\xi}A=(t^*-s^*)B \ , \mathscr{L}_X \theta =dB \}$.
By abuse of language, we will call a multiplicative vector field $(\xi,X)$ a weak infinitesimal symmetry, if there exists a one-form $B$
such that $(\xi,X,B)$ is such a symmetry.\\

\begin{rem}
Given a weak infinitesimal symmetry, the 1-form $B$ is unique up to addition of the pullback  under $\Pi$ of a closed 1-form on $M$.
\end{rem}

We first characterize the weak infinitesimal symmetries in terms of the multi\-symplectic notion of ``Hamiltonian vector fields'', which
will be considered in more detail in Section 3.

\begin{prop}
Let $(\xi,X)$ be a multiplicative vector field on the Lie groupoid $P\rightrightarrows Y$. Then $(\xi,X)$ is a weak infinitesimal symmetry if and only if $X$ projects to a ``Hamiltonian vector field'', i.e., there exists $\alpha\in\Omega^1(M)$ such that $d\alpha=\iota_{\Pi_{*}X}\omega$.
\end{prop}

\begin{proof}
Recall that $\xi$ can be written as $\widetilde{X_2}+(t^*-s^*)f\cdot\mathfrak{z}$. Denote by $\bar{X}\in \mathfrak{X}(M)$ the vector field upon 
which $X$ projects. We have:

\begin{equation*}
\begin{split}
\mathscr{L}_{\xi} A &= d(\iota_{\xi} A)+\iota_{\xi}dA \\
&=d((t^*-s^*)f)+\iota_{\widetilde{X_2}}(t^*-s^*)\theta \\
&=(t^*-s^*)(df+\iota_X \theta)
\end{split}
\end{equation*}

therefore, the form $B\in\Omega^1(Y)$ such that $\mathscr{L}_{\xi} A=(t^*-s^*)B$ must take the form:

\begin{equation*}
B=df+\iota_X \theta +\Pi^*\alpha
\end{equation*}

with $\alpha\in\Omega^1(M)$. Next we have on one hand $\mathscr{L}_{X}\theta = d(\iota_X \theta)+\Pi^*\iota_{\bar{X}}\omega$ 
and on the other hand $\mathscr{L}_{X}\theta=dB$ so:

\begin{equation*}
\begin{split}
\Pi^*\iota_{\bar{X}}\omega&=dB - d(\iota_X \theta) \\
&=d(\iota_X \theta +\Pi^*\alpha +df) -d(\iota_X \theta) \\
&= \Pi^*d\alpha
\end{split}
\end{equation*}

therefore $\bar{X}$ is Hamiltonian. The converse is easily seen to be true, finishing the proof.
\end{proof}

From the preceding proof we easily obtain the following.

\begin{cor}
Let $(\xi,X,B)$ be a weak infinitesimal symmetry. Write $\xi=\widetilde{X_2}+(t^*-s^*)g\cdot\mathfrak{z}$, for a certain $g\in C^{\infty}(Y)$. 
Then there exists $\alpha\in\Omega^1(M)$ such that:

\begin{equation*}
d\alpha = \iota_{\Pi_*X}\omega \hspace{0.2cm} \hbox{and} \hspace{0.2cm}  B=\iota_X\theta+dg+\Pi^*\alpha \, .
\end{equation*}
\end{cor}

\begin{rem}
The space of weak infinitesimal symmetries is also considered in \cite{2020arXiv200312874K}, where they are referred to as \textit{vector fields that preserve} $(A,\theta)$. Our Proposition 2.14 below corresponds to 
Proposition 3.17 and Corollary 3.18 in \cite{2020arXiv200312874K}.
\end{rem}

As in the case of multiplicative vector fields, the space of weak infinitesimal symmetries can be viewed as the degree-0 term of a strict Lie 2-algebra.
Consider the two-term complex of vector spaces given by

\begin{equation*}
\mathfrak{X}_V(Y)\oplus C^{\infty}(Y) \xrightarrow{\eta} \mathscr{V}_{(P,Y)}
\end{equation*}
with $\eta(Z,h)=(\widetilde{Z_2}+(t^*-s^*)h\cdot\mathfrak{z},Z,\iota_Z \theta +dh)= (\overset{\circ}{\eta}(Z,h), \iota_Z\theta+dh)$. 
We define a bracket on $\mathscr{V}_{(P,Y)}$ by $[(\xi,X),B),(\xi',X'),B']=(([\xi,\xi'],[X,X']),\mathscr{L}_X B' - \mathscr{L}_{X'} B)$, and the 
natural induced action of weak infinitesimal symmetries on $\mathfrak{X}_V (Y)\oplus C^{\infty} (Y)$ (i.e. the forms $B$ do not act). 
We then have the following:

\begin{prop}
The weak infinitesimal symmetries
\begin{equation*}
\mathfrak{X}_V(Y)\oplus C^{\infty}(Y)\xrightarrow{\eta} \mathscr{V}_{(P,Y)}
\end{equation*}
form a Lie algebra crossed module, hence a strict Lie 2-algebra. Moreover, this Lie 2-algebra 
maps naturally to the Lie 2-algebra of multiplicative vector fields (upon forgetting the 1-form component
of a weak symmetry). 
\end{prop}

\begin{proof}
We first show that $\mathfrak{X}_V(Y)\oplus C^{\infty}(Y)\xrightarrow{\eta} \mathscr{V}_{(P,Y)}$ is a Lie algebra crossed module.
It then follows that this two-term complex forms indeed  a strict Lie 2-algebra.\\

For $(Z,h)\in \mathfrak{X}_V(Y)\oplus C^{\infty}(Y)$, we show that the conditions of Definition 2.9 are satisfied for 
$\overset{\circ}{\eta}(Z,h)$ and $B=\iota_Z\theta +dh$. In fact, we have:

\begin{equation*}
\begin{split}
\mathscr{L}_{\widetilde{Z_2}+(t^*-s^*)h\cdot\mathfrak{z}} A &= d(\iota_{(t^*-s^*)h\cdot\mathfrak{z}} A)+\iota_{\widetilde{Z_2}}dA \\
&=(t^*-s^*)(dh+\iota_Z \theta)=(t^*-s^*)B \, .
\end{split}
\end{equation*}

Moreover, since $Z$ is a vertical vector field, $\iota_Z d\theta=\iota_Z \Pi^*\omega=0$. Thus we have:

\begin{equation*}
\begin{split}
\mathscr{L}_{Z}\theta&=d(\iota_Z\theta)+ \iota_Z d\theta \\
&=dB \, .
\end{split}
\end{equation*}

In view of Proposition \ref{MVF=Lie2}, the only additional equation we have to verify is the ``form part'' of equation (A2) in the definition of a crossed module
 in the Appendix. This boils down to checking that, for $(\xi,X,B)\in\mathscr{V}_{(P,Y)}$ and $(Z,h)\in \mathfrak{X}_V(Y)\oplus C^{\infty}(Y)$

\begin{equation*}
\iota_{[X,Z]}\theta +d(X(h)-Z(g)-\theta(X,Z))=\mathscr{L}_X(\iota_Z\theta+dh)-\mathscr{L}_Z(\iota_X\theta+dg+\Pi^*\alpha) \, .
\end{equation*}

The last equation follows easily upon using the general formulas $\xi=\widetilde{X_2}+(t^*-s^*)g$ and $B=\iota_X\theta+dg+\Pi^*\alpha$.\\

We conclude with

\begin{equation*}
\begin{split}
\mathscr{L}_X(\iota_Z\theta+dh)-\mathscr{L}_Z(B)&=d(X(h)-Z(g)-\theta(X,Z)) +\iota_Xd\iota_Z\theta-\mathscr{L}_Z\iota_X\theta \\
&=d(X(h)-Z(g)-\theta(X,Z)) +\iota_X\mathscr{L}_Z\theta -\mathscr{L}_Z\iota_X\theta \\
&=\iota_{[X,Z]}\theta +d(X(h)-Z(g)-\theta(X,Z)) \, .
\end{split}
\end{equation*}

Finally, the Lie 2-algebra $\mathfrak{X}_V(Y)\oplus C^{\infty}(Y)\xrightarrow{\eta} \mathscr{V}_{(P,Y)}$ is mapped to the Lie 2-algebra of
multiplicative vector fields on $P\rightrightarrows Y$ by the natural morphism $\Psi =(\Psi_1,\Psi_2)$ defined as follows
(here $(\xi,X,B)$ is an element of $\mathscr{V}_{(P,Y)}$):
\begin{equation*}
{\Psi_1}\vert_{\mathfrak{X}_V(Y)\oplus C^{\infty}(Y)}=\hbox{id}_{\mathfrak{X}_V(Y)\oplus C^{\infty}(Y)}, \,\, \Psi_1(\xi,X,B)=(\xi,X)\,\, \hbox{and}
\, \, \Psi_2=0   \,\, .
\end{equation*}

This finishes the proof.
\end{proof}

\begin{defa}
We denote the above \textit{Lie 2-algebra of weak infinitesimal symmetries}
by $\hbox{WSym}(P,Y):=\hbox{WSym}(P,Y;A, \theta)$, suppressing often the dependence on the connective structure $(A,\theta)$
notationally. 
\end{defa}

%
%

\section{A prequantisation map for integral 2-plectic manifolds}

Let us recall Kostant's prequantisation map in symplectic geometry (compare e.g. \cite{MR0294568}). For an integral symplectic manifold 
$(M,\omega)$, 
let ${q:P}\rightarrow M$ be a principal $S^1$-bundle with connection whose curvature equals the pullback of
$-{2\pi i}\omega$. For $f\in C^{\infty}(M)$, $X_f$ denotes the Hamiltonian vector field fulfilling $df=\iota_{X_f}\omega$, $\widetilde{X_f}$ its 
horizontal lift to $P$ and $\mathfrak{z}$ the fundamental vector field of the $S^1$-action on $P$, given by 
$\mathfrak{z}_p=\frac{d}{d\tau}\big\vert_{\tau=0} \, p\cdot \exp(i2\pi \, \tau)$.
Then the Lie algebra morphism $C^{\infty}(M)\rightarrow \mathfrak{X}^{S^1}(P)$, the $S^1$-invariant vector fields on $P$, given by $f\mapsto \widetilde{X_f}+q^*f\cdot\mathfrak{z}$ is the prequantisation map.\\

Its 2-plectic analogue will be developped in this section as a Lie 2-algebra morphism from the Lie 2-algebra of observables associated 
to a 2-plectic manifold (see Proposition 3.3 below) to a certain Lie 2-algebra of multiplicative vector fields on the global geometric 
prequantisation of an integral 2-plectic manifold $(M,\omega)$. Recall that this prequantisation is given by a lifting bundle gerbe associated to a principal $PU(H)$-bundle $\Pi:Y\rightarrow M$ with connective structure $(A,\theta)$ such that its 3-curvature equals $(-2\pi i)\omega$. 

We also fix a principal connection on $\Pi$ and denote, for a vector field $X\in\mathfrak{X}(M)$, the ensuing horizontal lift to $Y$ by $X^h$. 
Its curvature will be written $\Omega$, a $\mathfrak{pu}(H)$-valued 2-form on $Y$. Given two vector fields $Z$ and $Z'$ on $Y$, 
we have a vertical vector field $\Omega(Z,Z')$ that is in $y\in Y$ defined by the following formula:

\begin{equation*}
\Omega(Z,Z')|_y:=\frac{d}{dt}\Big\vert_{t=0} y\cdot \exp(t\Omega_y(Z_y,Z'_y)) \, ,
\end{equation*}

where $y\cdot g$ denotes the action of $g\in PU(H)$ on $y\in Y$. Notably, we have for $X,X'\in\mathfrak{X}(M)$, the vertical vector field 
$\Omega(X^h,X'^h)$ on $Y$ fulfilling 

\begin{equation*}
[X,X']^h=[X^h,(X')^h]+\Omega(X^h,(X')^h)
\end{equation*}

(see, e.g., Section 19 of  \cite{MR2428390}, for this approach to the curvature of a principal bundle). 

Finally, we recall that given a $\Pi$-projectable vector field $Z$ on $Y$, we have introduced in Section 2 the notation $Z_2$ for the 
vector field on $Y^{[2]}$ that is given by $(y,y')\mapsto(Z_y,Z_{y'})$.\\

We now give the fundamental definitions regarding 2-plectic manifolds and the associated Lie 2-algebra of observables. Of course, 
Definitions 3.1 and 3.2, as well as Proposition 3.3 generalize immediately
to the multisymplectic (or n-plectic) case (compare \cite{MR2912195}), 
but we stick here to the 2-plectic case.

\begin{defa}
Let $M$ be a manifold and $\omega\in \Omega^3(M)$. We say that $(M,\omega)$ is a \textit{2-plectic manifold} if $\omega$ is closed and 
non-degenerate in the following sense: 
\begin{equation*}
\iota_X\omega=0 \hspace{0.3cm} \hbox{if and only if} \hspace{0.3cm}  X=0
\end{equation*}

for every vector field $X\in\mathfrak{X}(M)$.
\end{defa}

\begin{defa}
Let $(M,\omega)$ be a 2-plectic manifold, and $\alpha\in\Omega^1(M)$. We say that $\alpha$ is a \textit{Hamiltonian form} if there exists a vector 
field $X_{\alpha}$ such that:

\begin{equation*}
d\alpha=\iota_{X_\alpha}\omega.
\end{equation*}

We also say that $X_\alpha$ is the \textit{Hamiltonian vector field associated to $\alpha$}. We denote the space of Hamiltonian forms by 
$\Omega^1_{Ham}(M,\omega)$. 
\end{defa}

Note that, given $\alpha\in\Omega^1_{Ham}(M,\omega)$, the associated Hamiltonian vector field is unique, by the non-degeneracy of $\omega$. The converse is not true: the Hamiltonian form, given a Hamiltonian vector field, is only unique up to a closed 1-form.\\

The following natural algebraic structure on the space of Hamiltonian forms seems to be explicitely stated for the first time 
in \cite{MR2566161} (compare also \cite{MR2912195}). We refer to the appendix for relevant definitions and notations. 

\begin{prop}
Let $(M,\omega)$ be a 2-plectic manifold. The complex:

\begin{equation*}
C^{\infty}(M)\overset{d} \longrightarrow \Omega^1_{Ham}(M,\omega)
\end{equation*}

can be given the structure of a Lie-2 algebra by setting:

\begin{itemize}
\item $l_1(f)=df$  and  $l_1(\alpha)=0$
\item $l_2(\alpha,\beta)=\iota_{X_{\alpha}}\iota_{X_{\beta}}\omega$
\item $l_2(\alpha,f)=0$
\item $l_3(\alpha,\beta,\gamma)=\omega(X_{\alpha},X_{\beta},X_{\gamma})$
\end{itemize}

for all $\alpha,\beta,\gamma\in \Omega^1_{Ham}(M,\omega)$ and $f\in C^{\infty}(M)$.
\end{prop}

We will denote this Lie 2-algebra by $L_{\bullet}(M,\omega)$, and call it the \textit{Lie 2-algebra (of observables) of $(M,\omega)$}. We will also use the following notation for the bracket of forms: $\{\alpha,\beta\}:=l_2(\alpha,\beta)$.\\

We are now ready to describe the 2-plectic analogue of Kostant's prequantisation map in symplectic geometry. 

\begin{prop}
Let $(M,\omega)$ be a 2-plectic manifold with integral $\omega$, and $(P,Y)$ a global geometric prequantisation of 
$(M,\omega)$ with connective structure $(A,\theta)$. Then there is a Lie 2-algebra morphism 
$\Phi=(\Phi_1,\Phi_2): L_{\bullet}(M,\omega)\rightarrow \mathfrak{X}_{mult}(P,\Pi)$:

\begin{center}
\begin{tikzcd}
C^{\infty}(M)\arrow[r, "d"] \arrow[d,"\Phi_1" '] 
& \Omega^1_{Ham}(M,\omega) \arrow[dl,tail,"\Phi_2"] \arrow[d,"\Phi_1"] \\
\mathfrak{X}_V(Y)\oplus C^{\infty}(Y)    \arrow[r, "\overset{\circ}{\eta}" ']
& \mathfrak{X}_{mult}(P)
\end{tikzcd}
\end{center}

where:

\begin{itemize}
\item $\Phi_1(f)=(0,\Pi^*f)$
\item $\Phi_1(\alpha)=\left(\widetilde{(X^h_{\alpha})_2},X^h_{\alpha}\right)$
\item $\Phi_2(\alpha,\beta)=(\Omega(X^h_{\alpha},X^h_{\beta}),\theta(X^h_{\alpha},X^h_{\beta})+\Pi^*(\alpha(X_{\beta})-\beta(X_{\alpha})) \, .$
\end{itemize}
\end{prop}

\begin{proof}
The map $\Phi_1$ is a chain map because $X_{df}=0$ and $\overset{\circ}{\eta}((0,\Pi^*f))=0$, because $s^*\Pi^*f=t^*\Pi^*f$. 
Now we have to verify

\begin{equation*}
(*) \hskip1cm \Phi_1(\{\alpha,\beta\})=[\Phi_1(\alpha),\Phi_1(\beta)]+\overset{\circ}{\eta}(\Phi_2(\alpha,\beta))
\end{equation*}

and 

\begin{equation*}
(**) \hskip1cm    -\Phi_2(df,\alpha)=[\Phi_1(f),\Phi_1(\alpha)].
\end{equation*}

We recall that $X_{\{\alpha,\beta\}}=[X_{\alpha},X_{\beta}]$ and compute

\begin{equation*}
\begin{split}
\Phi_1(\{\alpha,\beta\})&=(\widetilde{([X_{\alpha},X_{\beta}]^h)_2},[X_{\alpha},X_{\beta}]^h) \\
&=([(\widetilde{X^h_{\alpha}})_2,(\widetilde{X^h_{\beta}})_2]+\widetilde{\Omega(X^h_{\alpha},X^h_{\beta})}_2+(t^*-s^*)\theta(X^h_{\alpha},X^h_{\beta})\cdot\mathfrak{z},[X^h_{\alpha},X^h_{\beta}]\\&+\Omega(X^h_{\alpha},X^h_{\beta}))
\end{split}
\end{equation*}

and

\begin{equation*}
\begin{split}
&[\Phi_1(\alpha),\Phi_1(\beta)]+\overset{\circ}{\eta}(\Phi_2(\alpha,\beta))=([\widetilde{(X^h_{\alpha})_2},\widetilde{(X^h_{\beta}})_2],[X^h_{\alpha},X^h_{\beta}])\\
&+(\widetilde{\Omega(X^h_{\alpha},X^h_{\beta})_2}+(t^*-s^*)\theta(X^h_{\alpha},X^h_{\beta})\cdot\mathfrak{z},\Omega(X^h_{\alpha},X^h_{\beta})).
\end{split}
\end{equation*}

Thus equation $(*)$ is shown. For $(**)$ we have

\begin{equation*}
\Phi_2(df,\alpha)=(0,\Pi^*(df(X_\alpha)))
\end{equation*}

and

\begin{equation*}
\begin{split}
[\Phi_1(f),\Phi_1(\alpha)]&=[(0,\Pi^*f),(\widetilde{(X^h_{\alpha})_2},X^h_{\alpha})] \\
&=-(0,X^h_{\alpha}(\Pi^*f)) \\
&=-(0,\Pi^*(df(X_\alpha))).
\end{split}
\end{equation*}

The verification of the fourth equation of Definition A.5 is a long, but straightforward computation, whose crucial poins we will explain 
in the sequel. First of all, we have terms of the type:

\begin{equation*}
\begin{split}
\Phi_2(\{\alpha,\beta\},\gamma)&=(\Omega([X_{\alpha},X_{\beta}]^h,X^h_{\gamma}),\theta([X_{\alpha},X_{\beta}]^h,X^h_{\gamma}) \\
&+
\Pi^*(\{\alpha,\beta\}(X_{\gamma})-\gamma([X_{\alpha},X_{\beta}]))) \\
&=(\Omega([X_{\alpha},X_{\beta}]^h,X^h_{\gamma}),\theta([X^h_{\alpha},X^h_{\beta}],X^h_{\gamma})\\
&+\theta(\Omega(X_{\alpha},X_{\beta}),X^h_{\gamma})+g(\alpha,\beta,\gamma))
\end{split}
\end{equation*}

where we set

\begin{equation*}
g(\alpha,\beta,\gamma):=\Pi^*(\{\alpha,\beta\}(X_{\gamma})-\gamma([X_{\alpha},X_{\beta}])).
\end{equation*}

Furthermore, we have

\begin{equation*}
\begin{split}
-\gamma([X_{\alpha},X_{\beta}])&=-X_{\alpha}(\gamma(X_{\beta}))+X_{\beta}(\gamma(X_{\alpha}))+d\gamma(X_{\alpha},X_{\beta}) \\
&=-X_{\alpha}(\gamma(X_{\beta}))+X_{\beta}(\gamma(X_{\alpha}))+\omega(X_{\alpha},X_{\beta},X_{\gamma})
\end{split}
\end{equation*}

and

\begin{equation*}
\{\alpha,\beta\}(X_{\gamma})=-\omega(X_{\alpha},X_{\beta},X_{\gamma}).
\end{equation*}

Thus we obtain

\begin{equation*}
g(\alpha,\beta,\gamma)=\Pi^*(-X_{\alpha}(\gamma(X_{\beta}))+X_{\beta}(\gamma(X_{\alpha}))).
\end{equation*}

Finally, we have on the RHS of the fourth equation of Definition A.5 terms of the type
\begin{equation*}
\begin{split}
[\Phi_1(\alpha),\Phi_2(\beta,\gamma)]
&=[(\widetilde{(X^h_{\alpha})_2},X^h_{\alpha}), (\Omega(X^h_{\beta},X^h_{\gamma}),\theta(X^h_{\beta},X^h_{\gamma})+\Pi^*(\beta(X_{\gamma})-\gamma(X_{\beta})))] \\
&=(\widetilde{(X^h_{\alpha})_2},X^h_{\alpha}) . (\Omega(X^h_{\beta},X^h_{\gamma}),\theta(X^h_{\beta},X^h_{\gamma})+\Pi^*(\beta(X_{\gamma})-\gamma(X_{\beta}))) \\
&=([X^h_{\alpha},\Omega(X^h_{\beta},X^h_{\gamma})],X^h_{\alpha}(\theta(X^h_{\beta},X^h_{\gamma})+\Pi^*(\beta(X_{\gamma})-\gamma(X_{\beta})))\\
&-\theta(X^h_{\alpha},\Omega(X^h_{\beta},X^h_{\gamma}))).
\end{split}
\end{equation*}\\

Let us now explain why the LHS and the RHS of the fourth equation of Definition A.5 coincide here. Firstly, 
the terms with $g(\alpha,\beta,\gamma)$ cancel out with the terms of the type $X^h_{\alpha}(\Pi^*(\beta(X_{\gamma})-\gamma(X_{\beta}))$.
Secondly, $l_3'(\Phi_1(\alpha),\Phi_1(\beta),\Phi_1(\gamma))=0$ since $l_3'=0$, and 
\begin{equation*}
\Phi_1(l_3(\alpha,\beta,\gamma))=(0,\Pi^*(\omega(X_{\alpha},X_{\beta},X_{\gamma})))=(0,{(d\theta)}(X^h_{\alpha},X^h_{\beta},X^h_{\gamma})) \, .
\end{equation*}
Developping the $d\theta$-term and comparing with the terms of the type $\theta([{X_{\alpha},X_{\beta}]^h,}$
$X^h_{\gamma})$ resp.
$X^h_{\alpha}(\theta(X^h_{\beta},X^h_{\gamma}))$ show that with respect to the discussed contributions LHS and RHS are equal.
We next observe that the term $\theta(X^h_{\alpha},\Omega(X^h_{\beta},X^h_{\gamma}))$ and its signed cyclic permutations
appear identically on both sides of the equation.

To conclude, we start by observing that 
\begin{equation*}
[X^h_{\alpha},\Omega(X^h_{\beta},X^h_{\gamma})]=X^h_{\alpha}(\Omega(X^h_{\beta},X^h_{\gamma})) \,  , 
\end{equation*}

where the last term is interpreted as the vertical vector field associated to the 
$\mathfrak{pu}(H)$-valued map $X^h_\alpha(\Omega(X^h_\beta,X^h_\gamma)$). To show this, denote by $C\in\Omega^1(Y)\otimes\mathfrak{pu}(H)$ a principal connection on $\Pi:Y\rightarrow M$ with fixed curvature $\Omega$. Then $C(X_\alpha^h)=0$ because $X_\alpha^h$ is horizontal and $dC(X_\alpha^h,\Omega(X_\beta^h,X_\gamma^h))=0$ because $dC=-\frac{1}{2}[C,C]+\Omega$ and $\Omega$ is horizontal. So we obtain
\begin{equation*}
\begin{split}
C([X_\alpha^h,\Omega(X_\beta^h,X_\gamma^h)])&=-dC(X_\alpha^h,\Omega(X_\beta^h,X_\gamma^h))\\
&+X_\alpha^h(C(\Omega(X_\beta^h,X_\gamma^h))-\Omega(X_\beta^h,X_\gamma^h)(C(X_\alpha^h)) \\
&=X_\alpha^h(\Omega(X_\beta^h,X_\gamma^h)).
\end{split}
\end{equation*}

(Note that $\Omega(X_\beta^h,X_\gamma^h)$ is here  interpreted both as a vertical vector field and a $\mathfrak{pu}(H)$-valued map on $Y$). Observing that $[X_\alpha^h,\Omega(X_\beta^h,X_\gamma^h)]$ is vertical thus yields the formula.
\vspace{0.5cm}

Since the curvature $\Omega$ is a horizontal form, we have 
$\Omega([X_{\alpha},X_{\beta}]^h,X^h_{\gamma})=$
\newline\noindent $\Omega([X^h_{\alpha},X^h_{\beta}],X^h_{\gamma})$
and, by the Bianchi identity, we have $d\Omega(X^h_{\alpha},X^h_{\beta},X^h_{\gamma})=0$.
It follows that the remaining terms on both sides of the fourth equation of Definition A.5 indeed coincide.     \end{proof}

\vspace{0.5cm}

In symplectic geometry, restricting the target space of Kostant's prequantisation map to the vector fields that preserve the connection (which are generally referred to as infinitesimal quantomorphisms), gives a Lie algebra isomorphism. We now 
show the 2-plectic analogue: since the map $\Phi_1$ in Proposition 3.4 actually lies in $\mathscr{V}_{(P,Y)}$, we can "restrict" this Lie 2-algebra morphism to the weak infinitesimal symmetries, which then gives a quasi-isomorphism. 

\begin{thm}
The Lie 2-algebra morphism of Proposition 3.4 refines to a quasi-isomorphism  of Lie 2-algebras 
$\Phi=(\Phi_1,\Phi_2): L_{\bullet}(M,\omega)\rightarrow \hbox{WSym}(P,Y)$ as follows

\begin{center}
\begin{tikzcd}
C^{\infty}(M)\arrow[r, "d"] \arrow[d,"\Phi_1" '] 
& \Omega^1_{Ham}(M,\omega) \arrow[dl,tail,"\Phi_2"] \arrow[d,"\Phi_1"] \\
\mathfrak{X}_V(Y)\oplus C^{\infty}(Y) \arrow[r, "\eta" ']
& \mathscr{V}_{(P,Y)}
\end{tikzcd}
\end{center}

where

\begin{itemize}
\item $\Phi_1(f)=(0,\Pi^*f)$
\item $\Phi_1(\alpha)=\left(\widetilde{(X^h_{\alpha})_2},X^h_{\alpha},\iota_{X^h_{\alpha}}\theta +\Pi^*\alpha \right)$
\item $\Phi_2(\alpha,\beta)=\left(\Omega(X^h_{\alpha},X^h_{\beta}),\theta(X^h_{\alpha},X^h_{\beta})+\Pi^*(\alpha(X_{\beta})-\beta(X_{\alpha}))\right) \, .$
\end{itemize}

\end{thm}

\begin{proof}
First we show the $\Phi_1$ map indeed takes values in $\mathscr{V}_{(P,Y)}$. Denote $B_{\alpha}=\iota_{X^h_{\alpha}}\theta +\Pi^*\alpha$ for $\alpha\in\Omega^1_{Ham}(M,\omega)$. Because $(t^*-s^*)\Pi^*\alpha=0$ it is clear that $\mathscr{L}_{\widetilde{(X_{\alpha}^h})_2}A=(t^*-s^*)B_{\alpha}$. Moreover:

\begin{equation*}
\begin{split}
\mathscr{L}_{X^h_{\alpha}}\theta &= d\iota_{X^h_{\alpha}} \theta + \iota_{X^h_{\alpha}}d\theta \\
&=d(B_{\alpha}-\Pi^*\alpha)+\Pi^*\iota_{X_{\alpha}}\omega \\
&=dB_{\alpha}-\Pi^*(d\alpha -\iota_{X_{\alpha}}\omega )\\
&=dB_{\alpha}.
\end{split}
\end{equation*} 

We obviously have $\Phi_1(df)=(0,0,\Pi^*df)=\eta(0,\Pi^*f)=\eta(\Phi_1(f))$.\\

To verify that $\Phi$ defines a Lie 2-algebra morphism, the only additional equality to be satisfied is the "form part" of the 
second equation in Definition A.5. Concretely, we have to check the following equation:

\begin{equation*}
(*) \hspace{1cm} \iota_{X_{\{\alpha,\beta\}}^h}\theta +\Pi^*\{\alpha,\beta\}=
\end{equation*}
\begin{equation*}
\mathscr{L}_{X^h_{\alpha}}B_{\beta} - \mathscr{L}_{X^h_{\beta}}B_{\alpha} + \iota_{\Omega(X^h_{\alpha},X^h_{\beta})}\theta + d(\theta(X_{\alpha}^h,X_{\beta}^h)+\Pi^*(\alpha(X_{\beta})-\beta(X_{\alpha}))) \, .
\end{equation*}

We start with the first term of the LHS of $(*)$:
\begin{equation*}
(I) \hskip1cm \iota_{X_{\{\alpha,\beta\}}^h}\theta =\iota_{[X^h_\alpha,X^h_\beta]}\theta +\iota_{\Omega(X^h_\alpha,X^h_\beta)}\theta \, . 
\end{equation*}

We rewrite now the first two terms of the RHS of $(*)$, via

\begin{equation*}
\mathscr{L}_{X_\alpha^h}B_\beta -\mathscr{L}_{X_\beta^h}B_\alpha =\mathscr{L}_{X_\alpha^h}(\iota_{X_\beta^h}\theta +\Pi^*\beta)-\mathscr{L}_{X_\beta^h}(\iota_{X_\alpha^h}\theta +\Pi^*\alpha) 
\end{equation*}

and 

\begin{equation*}
\begin{split}
\mathscr{L}_{X_\alpha^h}\iota_{X_\beta^h}\theta -\mathscr{L}_{X_\beta^h}\iota_{X_\alpha^h}\theta&= \mathscr{L}_{X_\alpha^h}\iota_{X_\beta^h}\theta - d(\iota_{X_\beta^h}\iota_{X_\alpha^h}\theta) -\iota_{X_\beta^h}d\iota_{X_\alpha^h}\theta\\
&=\mathscr{L}_{X_{\alpha}^h}\iota_{X_{\beta}^h}\theta -d(\theta(X_{\alpha}^h,X_{\beta}^h)) -\iota_{X_\beta^h}\mathscr{L}_{X_\alpha^h}\theta +\iota_{X_\beta^h}\iota_{X_\alpha^h}d\theta\\
&=\iota_{[X_\alpha^h,X_\beta^h]}\theta-d(\theta(X_\alpha^h,X_\beta^h))-\Pi^*\iota_{X_\alpha}\iota_{X_\beta}\omega.
\end{split}
\end{equation*}

Furthermore, we calculate

\begin{equation*}
\begin{split}
\mathscr{L}_{X_\alpha^h}\Pi^*\beta - \mathscr{L}_{X_\beta^h}\Pi^*\alpha &= d\iota_{X_\alpha^h}\Pi^*\beta+\iota_{X_\alpha^h}d\Pi^*\beta -d\iota_{X_\beta^h}\Pi^*\alpha-\iota_{X_\beta^h}d\Pi^*\alpha \\
&= d\Pi^*(\beta(X_\alpha)-\alpha(X_\beta))+2\Pi^*\iota_{X_\alpha}\iota_{X_\beta}\omega.
\end{split}
\end{equation*}

We thus arrive at

\begin{equation*}
\mathscr{L}_{X_\alpha^h}B_\beta -\mathscr{L}_{X_\beta^h}B_\alpha=\iota_{[X_\alpha^h,X_\beta^h]}\theta-d(\theta(X_\alpha^h,X_\beta^h))+\Pi^*\{\alpha,\beta\} + d\Pi^*(\beta(X_\alpha)-\alpha(X_\beta)).
\end{equation*}

Reordering the terms in the preceding equation and employing equation $(I)$ immediately gives equation $(*)$.\\

Finally we show that this morphism defines a quasi-isomorphism.\\

We have that $\ker(\eta)=\{(Z,g)\in\mathfrak{X}_V(Y)\oplus C^{\infty}(Y) \ | \ Z=0, 
(t^*-s^*)g=0, dg=0 \}=\{\Pi^*f  \ | \  f\in C^{\infty}(M) \ \hbox{and} \ df=0\}$ implying  that the map $\Phi_1$ is an isomorphism between 
$\ker(d)$ and $\ker(\eta)$. 

We turn to proving  that $\Phi_1$ defines an isomorphism from 
$\Omega^1_{Ham}(M,\omega)/ \hbox{im}(d)$ to $\mathscr{V}_{(P,Y)}/ \Ima(\eta)$. 
Since a multiplicative vector field $(\xi,X)$ is a weak infinitesimal symmetry if and only if $X$ projects to a Hamiltonian vector field, it is clear that 
there is an $\alpha\in \Omega^1_{Ham}(M,\omega)$ and a $Z\in\mathfrak{X}_V(Y)$ such that $X-Z=X_\alpha^h$. Since $\xi=\widetilde{X_2}+(t^*-s^*)g\cdot\mathfrak{z}$ for a certain $g\in C^{\infty}(Y)$, we have for all $f\in C^{\infty}(M)$ that $\xi-\widetilde{Z_2} -(t^*-s^*)(g+\Pi^*f)\cdot\mathfrak{z}=\widetilde{X_\alpha^h}$. Now, if $B\in \Omega^1(Y)$ such that $\mathcal{L}_X\theta=dB$, it can be written as $B=\iota_X\theta+dg+\Pi^*(\alpha+df)$, 
and then $B-\iota_Z\theta-dg=\iota_X\theta+\Pi^*(\alpha+df)$. Thus we have $(\xi,X,B)-\eta(Z,g+\Pi^*f)=\Phi_1(\alpha+df)$. 
Unicity of $\alpha$ (up to differentials of functions) now yields the desired isomorphism.
\end{proof}

\begin{rem}
We observe that the quasi-isomorphism $\Phi$ of the preceding theorem is injective but, of course, is far from being surjective. 
\end{rem}

\begin{rem}
 Theorem 5.1 of \cite{2020arXiv200312874K} shows that $L_{\bullet}(M,\omega)$ and $\hbox{WSym}(P,Y)$ are quasi-isomorphic
 via the existence of an ``invertible butterfly''. Our techniques are completely different and yield an explicit map of Lie 2-algebras
 that turns out to be a quasi-isomorphism. From the point of view of \cite{MR3241135},
the preceding theorem is a ``globalisation'' of the 2-plectic specialisation of Theorem 4.2.2 there concerning 
prequantisation in n-plectic geometry. There the authors give a quasi-isomorphism between the Lie n-algebra of 
observables of a n-plectic manifold, and what they call the \textit{infinitesimal quantomorphisms}, defined via 
a \v{C}ech-Deligne complex.
We do not rely on open coverings as used there but formulate the algebraic objects in a global way by 
first associating a global geometric prequantisation to an integral 2-plectic manifold.
\end{rem}

\vspace{0.5cm}

\begin{defa}
Let $(M,\omega)$ be a 2-plectic manifold with integral $\omega$, and $(P,Y)$ a global geometric prequantisation of 
$(M,\omega)$. The Lie 2-algebra morphism $\Phi=(\Phi_1,\Phi_2)$  from $L_{\bullet}(M,\omega)$ to 
$\hbox{WSym}(P,Y;A,\theta)$
of the above theorem is called the \textit{prequantisation map of $(M,\omega)$}.\\
\end{defa}

%
%

\begin{appendices}

\section{Lie 2-algebras and their morphisms}

In this appendix we assemble the definitions regarding Lie 2-algebras and crossed modules 
of Lie algebras needed in this article. For a more complete treatment see \cite{MR2068522}.

\begin{defa}
Let $L_{\bullet}:=L_{-1}\rightarrow L_0$ be a 2-term complex of vector spaces. A \textit{Lie 2-algebra structure on $L_{\bullet}$}
 consists in (multi)linear graded antisymmetric maps $\{l_k \, \vert \, 1\leq k \leq 3  \}$ with the degree of $l_k$ being equal to $2{-}k$:

\begin{itemize}
\item $l_1:L_{-1}\rightarrow L_0$
\item $l_2:\Lambda^{2} L_{\bullet}\rightarrow L_{\bullet}$ 
\item $l_3:\Lambda^3 L_0\rightarrow L_{-1}$
\end{itemize}

such that the following equations hold, for $x,y,z,t\in L_0$ and $u,v\in L_{-1}$ :

\begin{equation*}
l_1(l_2(x,u))=l_2(x,l_1(u)) \hskip0.3cm \hbox{and} \hskip0.3cm l_2(l_1(u),v)=l_2(u,l_1(v))
\end{equation*}

\begin{equation*}
l_1(l_3(x,y,z))+l_2(l_2(x,y),z)-l_2(l_2(x,z),y)+l_2(l_2(y,z),x)=0
\end{equation*}

\begin{equation*}
l_3(l_1(u),x,y)+l_2(l_2(x,y),u)-l_2(l_2(x,u),y)+l_2(l_2(y,u),x)=0
\end{equation*}

\begin{equation*}
\begin{split}
& l_3(l_2(x,y),z,t)-l_3(l_2(x,z),y,t)+l_3(l_2(x,t),y,z) \\
& +l_3(l_2(y,z),x,t)-l_3(l_2(y,t),x,z)+l_3(l_2(z,t),x,y) \\
& = l_2(l_3(x,y,z),t)-l_2(l_3(x,y,t),z)+l_2(l_3(x,z,t),y)-l_2(l_3(y,z,t),x) \, .
\end{split}
\end{equation*}

\end{defa}

\begin{defa}
A \textit{Lie algebra crossed module} is given by two Lie algebras $\mathfrak{h}$ and $\mathfrak{g}$, a Lie algebra morphism 
$\eta:\mathfrak{h}\rightarrow \mathfrak{g}$ and an action $\vartheta:\mathfrak{g}\times\mathfrak{h}\rightarrow\mathfrak{h}$ by 
derivations such that\\

(A1) \hspace{2cm} $\vartheta(\eta(v),w)=[v,w]_{\mathfrak{h}}$\\

(A2) \hspace{2cm} $\eta(\vartheta(X,w))=[X,\eta(w)]_{\mathfrak{g}}$\\

for all $v,w\in\mathfrak{h}$ and $X\in\mathfrak{g}$. One denotes such a crossed module as a quadruple 
$(\mathfrak{h},\mathfrak{g},\eta,\vartheta)$.

\end{defa}

\begin{rem}
A Lie algebra crossed module is naturally given the structure of a \textit{strict} Lie 2-algebra (i.e. a Lie 2-algebra with third bracket identically 0) by setting for $X,Y\in \mathfrak{g}$, $v\in\mathfrak{h}$:

\begin{itemize}
\item $L_0:=\mathfrak{g}$, $L_{-1}:=\mathfrak{h}$
\item $l_1:=\eta$
\item $l_2(X,Y):=[X,Y]_{\mathfrak{g}}$, $l_2(X,v):=\vartheta(X,v)$
\item $l_3:=0$
\end{itemize}
Note that one often writes $[,]$ for the operation $l_2$ of the Lie 2-algebra associated to a Lie algebra crossed module.

\end{rem}

\begin{ex} We will now describe the crossed module structure on the multiplicative vector fields on a groupoid, which is used in Section 2. We will only give the general construction without any proofs, all the details can be found in \cite{2016arXiv160903944B} (see also \cite{MR3985597}). Let indeed $G_1\rightrightarrows G_0$ be a Lie groupoid and $\hbox{Lie}(G_1)\rightarrow G_0$ the associated Lie algebroid. Recall that 
$\hbox{Lie}(G_1)=\ker(s_*)|_{\varepsilon(G_0)}$. We note $s,t$ the source respectively the target map, $\varepsilon$ the unit map, $i$ the inverse map, and $R_g,L_g$ the right respectively left multiplication by an element $g\in G_1$. We have the right and left invariant vector fields on $G_1$ associated to a section 
$a\in \Gamma(G_0,\hbox{Lie}(G_1))$, given respectively by:

\begin{equation*}
\overrightarrow{a}(g)=(R_g)_{*}(a(t(g))) \hspace{0.3cm} \hbox{and}
\end{equation*}

\begin{equation*}
\overleftarrow{a}(g)=(L_g)_{*}(i_{*})(a(t(g)))        \, .
\end{equation*}

Then $\eta(a)=\overrightarrow{a}+\overleftarrow{a}$ is a multiplicative vector field on $G_1$. Given $\xi\in \mathfrak{X}_{mult}(G_1)$, then its action on the Lie algebroid sections is given by:

\begin{equation*}
[\xi,a]:=[\xi,\overrightarrow{a}]          \, ,
\end{equation*}

where the bracket on the right hand side is the bracket of vector fields on $G_1$ (recall that $\Gamma(G_0,\hbox{Lie}(G_1))$ is,
as a $C^{\infty}(G_0)$-module, isomorphic 
to the right-invariant vector fields on $G_1$). Thus the two-term complex of Lie algebras:

\begin{equation*}
\Gamma(G_0, \hbox{Lie}(G_1))\xrightarrow{\eta}\mathfrak{X}_{mult}(G_1)
\end{equation*}

together with the action defined above, and the usual bracket on $\Gamma(G_0,\hbox{Lie}(G_1))$ and 
$\mathfrak{X}_{mult}(G_1)$, is a Lie algebra crossed module, hence a strict Lie 2-algebra.
\end{ex}

Considering a Lie 2-algebra as a special case of a Lie $\infty$-algebra one obtains immediately the following

\begin{defa}
Let $(L_{\bullet}=L_{-1}\rightarrow L_0,\{l_k\})$ and $(L_{\bullet}' =L_{-1}'\rightarrow L_{0}' ,\{l_k'\})$ be two Lie 2-algebras. A \textit{Lie ${\infty}$-algebra morphism $L_{\bullet}\xrightarrow{\Phi} L_{\bullet}' $} is given by linear maps :
\begin{itemize}
\item $\Phi_1:L_{\bullet}\rightarrow L_{\bullet}'$
\item $\Phi_2:\Lambda^2 L_{\bullet}\rightarrow  L_{\bullet}'  $
\end{itemize}
with $\Phi_1$ of degree 0 and $\Phi_2$ of degree $-1$, such that for every $x,y,z\in L_0$ and $u\in L_{-1}$ :
\begin{equation*}
\Phi_1(l_1(u))=l_1'(\Phi_1(u))
\end{equation*}

\begin{equation*}
\Phi_1(l_2(x,y)) = l_2'(\Phi_1(x),\Phi_1(y)) + l_1'(\Phi_2(x,y))
\end{equation*}

\begin{equation*}
\Phi_1(l_2(u,x))  =l_2'(\Phi_1(u),\Phi_1(x)) +  \Phi_2(l_1(u),x)
\end{equation*}

\begin{equation*}
	\begin{split}
&\Phi_2(l_2(x,y),z)-\Phi_2(l_2(x,z),y)+\Phi_2(l_2(y,z),x)+\Phi_1(l_3(x,y,z))\\
&=l_2'(\Phi_1(x),\Phi_2(y,z)) -l_2'(\Phi_1(y),\Phi_2(x,z))+l_2'(\Phi_1(z),\Phi_2(x,y)) \\
&+l_3'(\Phi_1(x),\Phi_1(y),\Phi_1(z)).
	\end{split}
\end{equation*}
\end{defa}

\begin{rem}
A Lie $\infty$-algebra morphism between two Lie 2-algebras is, of course, also called a {\it Lie 2-algebra morphism}.
\end{rem}

\begin{rem}
Following a suggestion of Camille Laurent-Gengoux, we visualize a Lie 2-algebra morphism $L_{\bullet}\xrightarrow{\Phi} L'_{\bullet}$ as below:

\begin{center}
\begin{tikzcd}
L_{-1}\arrow[r,"l_1"] \arrow[d,"\Phi_1" '] 
& L_0 \arrow[dl,tail,"\Phi_2"] \arrow[d,"\Phi_1"] \\
L'_{-1} \arrow[r,"l'_1"']
& L'_0
\end{tikzcd}
\end{center}

Note that the outer square is a commutative diagram since $\Phi_1\circ l_1=l'_1\circ\Phi_1$. We underline that neither ``subdiagrams'' containing the diagonal commute, nor is $\Phi_2$ defined on the vector space $L_0$ (nor does the form of the diagonal arrow indicate any kind of injectivity).
\end{rem}

\begin{defa}
A cochain complex morphism is called a {\it quasi-isomorphism} if the induced map in cohomology is an isomorphism. 
A Lie $\infty$-algebra morphism between two Lie 2-algebras (or Lie $\infty$-algebras, in fact) is called a {\it quasi-isomorphism} 
if viewed as a cochain morphism it is a quasi-isomorphism.

\end{defa}

%
%


\section{The case of exact 2-plectic manifolds}

In this appendix, we expose the case of 2-plectic manifolds $(M,\omega)$ with a potential $\chi \in \Omega^2(M)$ for the 3-form 
$\omega$, i.e., $d\chi = \omega$. This case is important for applications in physics, compare, e.g., \cite{MR334772}, \cite{MR2559661} and \cite{MR3937870}.\\

Recall that for $P\xrightarrow{r}B$ and $P'\xrightarrow{r'}B$ two principal $S^1$-bundles over $B$ we may form a new bundle 
$P\otimes P'\rightarrow B$ defined by

\begin{equation*}
P\otimes P':=(P {{}_{r}\times_{r'}}P')/S^1,
\end{equation*}

where the quotient is taken so that $(p,p')\sim (p z,p' z^{-1})$ for $p\in P$, $p'\in P'$ and $z\in S^1$. We may also form the dual bundle 
$P^*$ defined by the same fiber bundle but provided with the action $p\cdot z=pz^{-1}$.\\

Let $Y\xrightarrow{\Pi}M$ be a surjective submersion and consider $Q\xrightarrow{r} Y$ a principal $S^1$-bundle, with  
connection 1-form $\widetilde{A}\in i\Omega^1(Q)$. We denote $F^{\widetilde{A}}\in i\Omega^2(Y)$ the curvature of this principal connection, 
so that $d\widetilde{A}=r^*F^{\widetilde{A}}$. Furthermore, we put $\Pi_k:Y^{[2]}\rightarrow Y$ for $k=1,2$, where $\Pi_k$ is the projection 
that omits the $k^{th}$ factor. Then 

\begin{equation*}
1\rightarrow Y\times S^1\rightarrow(\pi_1^*Q)^*\otimes \pi_2^*Q\rightarrow Y^{[2]}\rightarrow 1
\end{equation*}

is a $S^1$-central extension of the Lie groupoid $Y^{[2]}\rightrightarrows Y$, and therefore $\delta(Q):=(\pi_1^*Q)^*\otimes \pi_2^*Q$ yields a bundle 
gerbe. In fact (compare \cite{MR1794295}), a bundle gerbe has vanishing Dixmier-Douady class precisely when it is isomorphic (in an appropriate sense) 
to a bundle gerbe of the form $(\delta(Q),Y)$, and a choice of an isomorphism $P\rightarrow \delta(Q)$ is called a \textit{trivialisation} of the bundle gerbe $(P,Y)$.\\

Let us recall several observations concerning the Lie groupoid $\delta(Q)\rightrightarrows Y$. An element of $\delta(Q)$ is written as a quadruple 
$([q_1,q_2],y_1,y_2)$ where $q_i\in Q$, $(y_1,y_2)\in Y^{[2]}$ such that $r(q_1)=y_1$, $r(q_2)=y_2$, and $[q_1,q_2] $ is the class 
stemming from the equivalence relation:  $(q_1,q_2)\sim (q_1z,q_2z)$ for all $z\in S^1$. In the sequel, we denote such a quadruple
by the equivalence classe $[q_1,q_2]$. Then the structure maps of the Lie groupoid $\delta(Q)\rightrightarrows Y$ are given by

\begin{itemize}
\item $s([q_1,q_2])=r(q_2)$
\item $t([q_1,q_2])=r(q_1)$.
\end{itemize}

Note that for $[q_1,q_2] $,$[ q_3,q_4] $ in $\delta(Q)$, such that $s([q_1,q_2])=r(q_2)=r(q_3)=t([q_3,q_4])$, we may assume without loss of generality that 
$q_2=q_3$, since the equivalence classes are defined by the orbits of the $S^1$-action. Therefore the groupoid multiplication can be written as

\begin{equation*}
m([q_1,q_2],[q_2,q_3])=[q_1,q_3].
\end{equation*}

For $y\in Y$ we have furthermore

\begin{equation*}
\delta(Q)_{(y,y)}=\left((\pi_1^*Q)^*\otimes \pi_2^*Q\right)_{(y,y)}=Q_y^*\otimes Q_y.
\end{equation*}

Thus denoting by $\Delta:Y\rightarrow Y^{[2]}$ the diagonal inclusion, the bundle $\Delta^*(\delta(Q))$ is canonically trivialised. Taking for $y\in Y$ any
$q\in Q_y$, the canonical section $\varepsilon:Y\rightarrow \Delta^*(\delta(Q))$ is given by $\epsilon(y)=[q,q]$ and $\epsilon$ is taken as the unit map 
of the Lie groupoid 
$\delta(Q)\rightrightarrows Y$. Finally the inverse map is given as

\begin{equation*}
[q_1,q_2]^{-1}=[q_2,q_1].
\end{equation*}

We now describe a connective structure on $(\delta(Q),Y)$. Consider the projections $(\pi_1^*Q)^*\times_{Y^{[2]}}\pi_2^*Q\xrightarrow{p_k}Q$, 
where $p_k$ projects to the $k^{th}$ factor ($k=1,2$), and the 1-form $A=p_1^*\widetilde{A}-p_2^*\widetilde{A}$. Then $A$ defines a principal 
connection on $(\pi_1^*Q)^*\otimes \pi_2^*Q$. More precisely, we have:

\begin{lem}
We keep the notations from the previous considerations. Then
\begin{enumerate}[label=(\roman*)]
\item the connection $A$ on $\delta(Q)\rightarrow Y^{[2]}$ is multiplicative,
\item {a curving of this connection is given by the curvature $F^{\widetilde{A}}$ of $\widetilde{A}$.}
\end{enumerate}

Therefore $(A,F^{\widetilde{A}})$ is a connective structure on the bundle gerbe $(\delta(Q),Y)$. Moreover, the 3-curvature is identically zero
and thus $(\delta(Q),Y)$ has vanishing Dixmier-Douady class. 
\end{lem}

\begin{proof}
For $q_1,q_2,q_3\in Q$ and $v_{q_1},v_{q_2},v_{q_3}$ tangent vectors at the respective points, we compute
\begin{equation*}
\begin{split}
{}&(m^*A)_{[q_1,q_2],[q_2,q_3]}((v_{q_1},v_{q_2}),(v_{q_2},v_{q_3}))=A_{[q_1,q_3]}(v_{q_1},v_{q_3}) 
=\widetilde{A}_{q_1}(v_{q_1})-\widetilde{A}_{q_3}(v_{q_3})=\\
&\widetilde{A}_{q_1}(v_{q_1})-\widetilde{A}_{q_2}(v_{q_2})+\widetilde{A}_{q_2}(v_{q_2})-\widetilde{A}_{q_3}(v_{q_3}) 
=A_{[q_1,q_2]}(v_{q_1},v_{q_2})+A_{[q_2,q_3]}(v_{q_2},v_{q_3})=\\
&(\hbox{proj}_1^*A+\hbox{proj}_2^*A)_{[q_1,q_2] ,[q_2,q_3]}((v_{q_1},v_{q_2}),(v_{q_2},v_{q_3})),
\end{split}
\end{equation*}

where $\hbox{proj}_k:\delta(Q){{}_s\times_t}\delta(Q)\rightarrow \delta(Q)$ is the projection onto the $k^{th}$ factor, for $k=1,2$. 
Thus $A$ is multiplicative.\\

 Moreover, we have 
 \begin{equation*}
dA=p_1^*d\widetilde{A}-p_2^*d\widetilde{A}=p_1^*r^*F^{\widetilde{A}} -p_2^*r^*F^{\widetilde{A}} =(t^*-s^*)F^{\widetilde{A}}.
\end{equation*}

This shows that $F^{\widetilde{A}}$ provides a curving for $A$, and that $(A,F^{\widetilde{A}})$ is indeed a connective structure for $(\delta(Q),Y)$. Since $dF^{\widetilde{A}}=0$, the 3-curvature equals zero, and therefore $(\delta(Q),Y)$ has vanishing Dixmier-Douady class.
\end{proof}

We now construct the prequantisation map for an exact 2-plectic manifold. Let $(M,\omega)$ be a 2-plectic manifold with $\omega=d\chi$ an exact 3-form. Consider the trivial bundle $\Pi:Y:=M\times PU(H)\rightarrow M$, equipped with the trivial connection, that we denote here by $C\in\Omega^1(M\times PU(H))\otimes\mathfrak{pu}(H)$. Recall that $C$ is given by

\begin{equation*}
C_{(m,g)}(u_m,v_g)=(L_{g^{-1}})_{*g}(v_g),
\end{equation*}

i.e., $C$ is the pullback of the Maurer-Cartan form on $PU(H)$ via the projection $Y\rightarrow PU(H)$. \\

Observe that $Y^{[2]}=M\times PU(H)\times PU(H)$. We then have $\psi:M\times PU(H)\times PU(H)\rightarrow PU(H)$, $\psi(m,g,g')=g^{-1}g'$ (compare Lemma 1.10). We also define $Q:=M\times U(H)\rightarrow Y$. Then $\delta(Q):=(\pi_1^*Q)^*\otimes \pi_2^*Q$ is isomorphic to $\psi^*U(H)$, the pullback of the principal $S^1$-bundle $U(H)\rightarrow PU(H)$ by the map $\psi$. To see this, note that the elements of $\delta(Q)$ are triples $(m,[u,u'])$, where $m\in M$, $u,u'\in U(H)$, and the class $[u,u']$ is taken with respect to $(u,u')\sim (uz,u'z)$ for all $z\in S^1$. The elements of $\psi^*U(H)$ are triples $(m,g,u)$, where $m\in M$, $g\in PU(H)$ and $u\in U(H)$. Then the isomorphism $\delta(Q)\rightarrow \psi^* U(H)$ is given by

\begin{equation*}
(*) \hspace{1cm} (m,[u,u'])\mapsto (m,q(u),u^{-1}u') \, ,
\end{equation*}

with $U(H)\xrightarrow{q}PU(H)$ being the canonical projection. The map $(*)$ is well-defined and equivariant, and therefore 
an isomorphism of principal $S^1$-bundles. \\

Let $\widetilde{A}$ be the principal connection on $Q$ defined by the Maurer-Cartan form on $U(H)$, projected onto $i\mathbb{R}=\hbox{Lie}(S^1)$ via a splitting $\mathfrak{pu}(H)\rightarrow \mathfrak{u}(H)$ (see Lemma 1.6). Let $F^{\widetilde{A}}$ be the curvature of this connection. On the bundle 
$\delta(Q)\rightarrow Y^{[2]}$, we consider as above the connection $A=p_1^*\widetilde{A}-p_2^*\widetilde{A}$ (recall that $A$ is multiplicative by the preceding lemma). We set $\theta=\Pi^*\chi+F^{\widetilde{A}}$. Then $(A,\theta)$ is a connective structure on the bundle gerbe $(\delta(Q),Y)$, with 3-curvature 
$(-i2\pi) \omega$. For  a vector field $X\in\mathfrak{X}(M)$, the horizontal lift $X^h$ to $Y$ with respect to the conection $C$ is simply $(X,0)$, which we  denote again by $X$. Then we have $X_2=(X,0,0)$ and we continue to denote this vector field by $X$. The horizontal lift of a vector field $Z\in\mathfrak{X}(Y^{[2]})$ with respect to the connection $A$ will be denoted by $\widetilde{Z}$.\\

We conclude with an explicit description of the components of the Lie 2-algebra morphism of Theorem 3.5 in the exact case.
For the sake of better readability, we omit the symbol $\Pi^*$ for pullbacks of functions and differential forms with respect to the 
projection $\Pi:Y\rightarrow M$. Furthermore, given a vector field $V$ on a factor of a product $A\times B$, we denote its 
trivial extension to this product again by $V$. With these conventions, we obtain for $f\in C^{\infty}(M)$ and 
$\alpha,\beta \in  \Omega^1_{Ham}(M,d\chi)$:

\begin{itemize}
\item $\Phi_1(f)=(0,f)$
\item $\Phi_1(\alpha)=(\widetilde{X_\alpha},X_\alpha,\iota_{X_\alpha}\chi +\alpha)$
\item $\Phi_2(\alpha,\beta)=(0,\chi(X_\alpha,X_\beta)+\alpha(X_\beta)-\beta(X_\alpha))$.
\end{itemize}

\end{appendices}

%
%
\bibliography{two-prequantisation-5-7-21}
\bibliographystyle{plain}

\end{document}